\documentclass[10pt]{amsart}
\usepackage[utf8]{inputenc}
\usepackage{graphics}
\usepackage[all, cmtip]{xy}
\usepackage{amsthm,amsfonts,amssymb,amsmath,amsxtra}
\usepackage[all]{xy}
\usepackage{flexisym}
\usepackage{mathrsfs}
\usepackage{enumitem}
\usepackage[colorlinks]{hyperref}
\hypersetup{
  citecolor   = blue 
 }

\usepackage{tikz-cd}
\usepackage{blkarray}
\newcommand{\matindex}[1]{\mbox{\scriptsize#1}}

\newcommand{\xdownarrow}[1]{%
  {\left\downarrow\vbox to #1{}\right.\kern-\nulldelimiterspace}
}
\newcommand{\quash}[1]{}  

\newcommand{\M}{{\rm M}}
\newcommand{\loc}{{\rm loc}}
\newcommand{\spl}{{\rm spl}}
\newcommand{\bl}{{\rm bl}}
\newcommand{\Kra}{{\rm Kra}}
\newcommand{\nspl}{{\rm nspl}}
\newcommand{\NExc}{{\rm NExc}}
\newcommand{\Exc}{{\rm Exc}}

\newcommand{\ZZ}{\mathbb{Z}}
\newcommand{\PP}{\mathbb{P}}

\newcommand{\calF}{\mathcal{F}}
\newcommand{\calL}{\mathcal{L}}

\newcommand{\calM}{\mathcal{M}}

\newcommand{\calU}{\mathcal{U}}
\newcommand{\calN}{\mathcal{N}}
\newcommand{\calZ}{\mathcal{Z}}
\newcommand{\calY}{\mathcal{Y}}

\newcommand{\scrG}{\mathscr{G}}

\newcommand{\bb }{\langle}
\newcommand{\pp}{\rangle}

\newcommand{\rightarroweq}{\stackrel{\sim}{\rightarrow}}
\newcommand{\br}{\breve}

\usepackage{multicol}
\numberwithin{equation}{subsection}

\newtheorem{Theorem}{Theorem}[section]
\newtheorem{Remark}[Theorem]{Remark}
\newtheorem{Remarks}[Theorem]{Remarks}

\newtheorem{Proposition}[Theorem]{Proposition}
\newtheorem{Corollary}[Theorem]{Corollary}

\newtheorem{Main Theorem}[Theorem]{Main Theorem}
\newtheorem{Definition}[Theorem]{Definition}

\usepackage[colorlinks]{hyperref}
\makeatletter
\renewcommand*\env@matrix[1][*\c@MaxMatrixCols c]{%
  \hskip -\arraycolsep
  \let\@ifnextchar\new@ifnextchar
  \array{#1}}
\makeatother

\newif\ifgrading

\gradingfalse

\usepackage[colorlinks]{hyperref}
\hypersetup{linkcolor=black}
\usepackage{xcolor}
\usepackage{ wasysym }

\usepackage{tikz-cd}
\newcommand{\U}{{\mathcal U}}
\newcommand{\calG}{{\mathcal G}}

\newcommand{\Spec}{{\rm Spec \, } }

\usepackage{xcolor}

\newcommand{\Addresses}{{
		\bigskip
		\footnotesize
		\textsc{Department of Mathematics, Universität Münster, Münster, 48149, Germany}\par\nopagebreak
		\textit{E-mail address:} \texttt{io.zachos@uni-muenster.de}\\
		
	\textsc{Department of Applied Mathematics, University of Science and Technology Beijing,
			Beijing, 100083, China}\par\nopagebreak
		\textit{E-mail address:} \texttt{zhihaozhao@ustb.edu.cn}
}}	

\begin{document}
\title[Regular Ramified Unitary Rapoport--Zink Spaces]{The Basic Locus of Regular Ramified Unitary Rapoport--Zink Spaces at Vertex--Stabilizer Level}
	\date{}
	\author{I. Zachos and Z. Zhao}
		
	\begin{abstract}
We construct the Bruhat--Tits stratification of the reduced basic locus of regular ramified unitary Rapoport--Zink spaces of signature $(n\!-\!1,1)$ at vertex--stabilizer level. To study the Bruhat--Tits strata, we introduce strata models--simpler models that are \'{e}tale-locally isomorphic to each stratum. They admit two complementary characterizations: (i) as strict transforms under the blow-up of the local model at its worst point, and (ii) via a partial moduli description given by explicit linear-algebraic conditions; 
from these we deduce smoothness, explicit dimension formulas and irreducibility of the Bruhat--Tits strata.
	\end{abstract}
\maketitle	
	\tableofcontents

\section{Introduction}\label{Intro}
\subsection{} This work contributes to the theory of integral models of Shimura varieties by giving a concrete description of the reduced basic locus of regular ramified unitary Rapoport--Zink (RZ) spaces at a maximal vertex level with signature $ (n-1,1)$. Over roughly the last twenty years, the geometry of the basic locus of Shimura varieties has been explored in many settings both orthogonal and unitary and now forms a substantial literature; we refer to the introductions of \cite{Vol10}, \cite{Wu} and \cite{HZ} for more historical content and background. Beyond its intrinsic interest, it has driven progress on many important arithmetic intersection problems such as the Kudla--Rapoport conjecture in which the computation of intersection numbers of cycles in the basic locus forms a crucial input (see \cite{CZ1,CZ2,HLSY,LL,Yao,Luo2}). The explicit description of the basic locus has also contributed to the arithmetic fundamental lemma and the arithmetic transfer conjecture (see, for example,  \cite{RSZ,Zhang,LMZ,LRZ}).

For the general unitary group ${\rm GU}(n-1,1)$, the basic locus of the RZ space was first studied by Vollaard \cite{Vol10} and Vollaard--Wedhorn \cite{VW} at an inert prime with hyperspecial level. In these works it is shown that the basic locus admits a 
stratification by (generalized) Deligne--Lusztig varieties and the incidence of strata is governed by some Bruhat–Tits building. 
This is the so-called \textit{Bruhat--Tits (BT) stratification}. G\"ortz--He--Nie \cite{GHN} subsequently gave a classification of Shimura varieties whose basic locus admits a BT stratification. More recently, Muller \cite{Muller} extended the results of Vollaard–Wedhorn to arbitrary parahoric level.

For the ramified unitary case of signature $(n-1,1)$, which is the case of interest here, the basic locus was studied by Rapoport--Terstiege--Wilson \cite{RTW} at self-dual level, by Wu \cite{Wu} for the exotic smooth cases, and more recently by He--Luo--Shi \cite{HLS2}, who extended the results to all maximal vertex levels. These RZ spaces have an explicit moduli description and this is a key tool for the study of their basic locus. Moreover, except for the two exotic smooth cases studied in \cite{Wu}, they have bad reduction at every maximal vertex level. 
For the self-dual level, in order to resolve the singularities and obtain a semistable model, a variation of this moduli problem was obtained by Krämer \cite{Kr} where she added to the moduli problem an additional linear datum of a flag of the Hodge filtration with certain restrictive properties. This model is commonly referred to as the Krämer model (or splitting model). In the course of proving the Kudla–Rapoport conjecture in the self-dual case, He--Li--Shi--Yang \cite{HLSY} determined the basic locus for the Krämer model. 

For more general vertex levels the plot is thicker since the natural generalization of the Krämer model as suggested in the work of Pappas--Rapoport \cite{PR2} gives a non-flat model. However, in \cite{HLS2}, He--Luo--Shi proved that its flat closure has semistable reduction. In this paper we give a concrete description of the reduced basic locus of the corresponding Rapoport–Zink space for any maximal vertex lattice. We expect that our construction will have applications to the aforementioned arithmetic intersections problems.
\subsection{} 
Let us give some details. We work over a ramified extension $F/F_0$ of $p$-adic fields (with $p$ odd), residue field $k$, and uniformizers $\pi,\pi_0$ satisfying $\pi^2=\pi_0$. Fix an algebraic closure $\bar{k}$ of $k$, and write $\breve F$ for the completion of the maximal unramified extension of $F$ with ring of integers $O_{\breve F}$. Let $0\le h<\tfrac{n}{2}$ be the index of the maximal vertex level. (See Remark \ref{rk 3.4} (3) for the case $n=2m$ and $h=m$.)

Fix a supersingular hermitian $O_F$-module $(\mathbb{X},\iota_{\mathbb{X}},\lambda_{\mathbb{X}})$ over $\mathrm{Spec}\,\bar k$ of rank $n$ and type $2h$ (signature $(n-1,1)$) as the framing object (see §\ref{sec2.2}). For $S\in\mathrm{Nilp}\, O_{\breve{F}}$, define $\calN_n^{\mathrm{Kra}}(S)$ to consist of isomorphism classes of $(X,\iota,\lambda,\rho,\mathrm{Fil}_0(X),\mathrm{Fil}_0(X^\vee))$ where $(X,\iota,\lambda)$ is a hermitian $O_F$-module over $S$ of signature $(n-1,1)$ and type $2h$, $\rho:X_{\bar S}\to\mathbb{X}_{\bar S}$ is an $ O_F$-linear quasi-isogeny of height $0$ with $\rho^*(\lambda_{\mathbb{X},\bar S})=\lambda_{\bar S}$, and $\mathrm{Fil}_0(X)\subset\mathrm{Fil}(X)\subset\mathbb D(X)$ and $\mathrm{Fil}_0(X^\vee)\subset\mathrm{Fil}(X^\vee)\subset\mathbb D(X^\vee)$ are rank-one locally free direct summands satisfying the \textit{transition conditions}
    \[
\lambda (\operatorname{Fil}^0(X)) \subset \operatorname{Fil}^0(X^\vee), \quad (\text{resp.} \quad \lambda^\vee (\operatorname{Fil}^0(X^\vee))\subset \operatorname{Fil}^0(X))
     \]
and \textit{splitting conditions}
    \[  (\iota(\pi)+\pi) (\operatorname{Fil}(X)) \subset \operatorname{Fil}^0(X), \quad  (\iota(\pi)-\pi) (\operatorname{Fil}^0(X))=0
    \]  
    \[ (\text{resp}.\quad  (\iota(\pi)+\pi) (\operatorname{Fil}(X^\vee)) \subset \operatorname{Fil}^0(X^\vee), \quad  (\iota(\pi)-\pi) (\operatorname{Fil}^0(X^\vee))=0).
    \]
We write $\calN_n^{\mathrm{Kra}}$ for the resulting formal scheme over $\operatorname{Spf} O_{\breve F}$ and this is the Krämer RZ space. By the local model diagram, $\calN_n^{\mathrm{Kra}}$ is \'etale locally isomorphic to the \textit{naive splitting model} $\M^{\nspl,[2h]}_n$. As observed in \cite[Remark 3.3]{ZacZhao2} and later proved in \cite{HLS1}, $\M^{\nspl,[2h]}_n$, and so $\calN_n^{\mathrm{Kra}}$, fails to be flat in general. To remedy this non-flatness, one can either try to modify the moduli problem and obtain a flat model, or define the splitting model as the scheme theoretic closure of the generic fiber of $\M^{\nspl,[2h]}_n$ in $\M^{\nspl,[2h]}_n$. 

The first route was taken in \cite{ZacZhao2}. The corresponding RZ space  is defined analogously to \(\mathcal{N}_n^{\rm Kra}\), where we impose only the additional filtration \(\operatorname{Fil}^0(X)\subset \operatorname{Fil}(X)\) (omitting \(\operatorname{Fil}^0(X^\vee)\)) together with the corresponding splitting conditions. This space is flat, normal and \'etale locally isomorphic to the splitting model $\mathscr{M}^{\spl,[2h]}_n$ which has an explicit moduli-theoretic description (see the Appendix for the definition). Its BT stratification was studied in \cite{ZacZhao3}. 
 
The latter route was taken in \cite{HLS1}, where the authors defined $\M^{\spl,[2h]}_n$ to be the {\it splitting model} as the scheme-theoretic closure of the generic fiber $\M^{\nspl,[2h]}_n\otimes_{O_F}F$ in $\M^{\nspl,[2h]}_n$. The corresponding RZ space $\mathcal{N}^\spl_n$ is a linear modification of $\mathcal{N}_n^{\rm Kra}$ in the sense of \cite[\S 2]{P} and there is a local-model diagram connecting $\mathcal{N}^\spl_n$ to the splitting model $\M^{\spl,[2h]}_n$. By \cite[Theorem 1.3.1]{HLS1}, $\M^{\spl,[2h]}_n$ is flat with semi-stable reduction; hence the same holds for $\mathcal{N}^\spl_n$. Here, we want to mention that since $\M^{\spl,[2h]}_n$ is defined via scheme-theoretic closure, it does not come equipped with an explicit moduli-theoretic description a priori (though in special cases one can recover such a description; see Remark \ref{rk 3.4}). Furthermore, we want to highlight that the splitting model $\M^{\spl, [2h]}_n$ is actually the blow-up of the canonical local model $\M^{\loc, [2h]}_n$ described in \cite{Luo} along the worst point $*$, i.e.  the only closed Schubert cell that lies in the special fiber of the local model (see \S \ref{RSpl}). Denote by $\mathcal{N}_n$ the corresponding RZ space of the local model. The BT stratification of $\mathcal{N}_n$ was studied in \cite{HLS2}.

In \cite{ZacZhao2}, we define $\mathscr{M}^{\bl,[2h]}_n$ as the blow-up of the inverse image in $\mathscr{M}^{\spl,[2h]}_n$ of the worst point of $\M^{\loc, [2h]}_n$ under the projective forgetful morphism $\tau':\mathscr{M}^{\spl,[2h]}_n\to\M^{\loc, [2h]}_n$ (see the Appendix). In the Appendix, we also show that $\mathscr{M}^{\bl,[2h]}_n$ is isomorphic to $\M^{\spl,[2h]}_n$, thereby connecting the above splitting-model variants.

Next, to the triple $(\mathbb{X},\iota_{\mathbb{X}},\lambda_{\mathbb{X}})$ there exists a hermitian space $C$ over $F$ of dimension $n$. Consider a lattice $\Lambda\subset C$ and its dual lattice $\Lambda^\sharp$ with respect to the hermitian form on $C$. We call $\Lambda\subset C$ a vertex lattice if it satisfies 
\[
\pi \Lambda^{\sharp} \subset \Lambda \subset  \Lambda^{\sharp}.   
\]
We denote by $t(\Lambda):=\dim(\Lambda^\sharp/\Lambda)$ the type of a vertex lattice, which is an even integer (see \S \ref{sec2.2}). By abuse of notation, we will write $2t$ instead of $t(\Lambda)$. Let $L_{\mathcal{Z}}$ denote the set of all vertex lattices of type $2t \geq 2h$, and let $L_{\mathcal{Y}}$ denote the set of all vertex lattices of type $2t\leq 2h$. For each $\Lambda_1 \in L_{\mathcal{Z}}$ and $\Lambda_2 \in L_{\mathcal{Y}}$, we define closed subschemes $\mathcal{Z}(\Lambda_1)$ and $\mathcal{Y}(\Lambda_2^\sharp)$ of the special fiber of the RZ space $\overline{\calN}_{n}^\spl$ (see \S \ref{BT_strata}). The main result of the paper is the following theorem which is proved in \S \ref{BTstrat.}. 

\begin{Theorem}\label{Thm BT intro}
The Bruhat-Tits stratification of the reduced basic locus is
\begin{equation}\label{BTstrintro}
      \calN_{n, {\rm red}}^\spl = \left( \bigcup_{\Lambda_1 \in L_{\mathcal{Z}} }\mathcal{Z}(\Lambda_1) \right) \cup \left( \bigcup_{\Lambda_2 \in L_{\mathcal{Y}}} \mathcal{Y}(\Lambda_2^\sharp) \right).
  \end{equation}
\begin{enumerate}
  \item  
   These strata satisfy the following inclusion relations:
  \begin{itemize}
    \item[(i)] For any $\Lambda_1, \Lambda_2 \in L_{\mathcal{Z}}$ of type greater than $2h$, if 
    $  \Lambda_1 \subseteq \Lambda_2$ then $\mathcal{Z}(\Lambda_2) \subseteq \mathcal{Z}(\Lambda_1)$.    
    \item[(ii)] For any $\Lambda_1, \Lambda_2 \in  L_{\mathcal{Y}}$ of type less than $2h$, if  
    $  \Lambda_1 \subseteq \Lambda_2$ then $\mathcal{Y}(\Lambda^{\sharp}_1) \subseteq \mathcal{Y}(\Lambda^{\sharp}_2)$.  

    \item[(iii)] For any $\Lambda_1\in L_{\mathcal{Z}}$ of type greater than $2h$, $\Lambda_2 \in  L_{\mathcal{Y}}$ of type less than $2h$, $  \Lambda_1 \subseteq \Lambda_2$ if and only if the intersection $\mathcal{Z}(\Lambda_1) \cap \mathcal{Y}(\Lambda_2^\sharp)$ is non-empty.
  \end{itemize}

  \item In the following, assume that $\Lambda, \Lambda'$ are vertex lattices of type $ 2t$ with $t \neq h$, and $\Lambda_0, \Lambda'_0$ are vertex lattices of type $2t$ with $ t=h$. 
  \begin{itemize}
    \item[(i)] The intersection $\mathcal{Z}(\Lambda) \cap \mathcal{Z}(\Lambda')$ (resp. $\mathcal{Y}(\Lambda^\sharp) \cap \mathcal{Y}(\Lambda^{\prime \sharp)}$) is non-empty if and only if $\Lambda'' = \Lambda + \Lambda'$ (resp. $\Lambda''=\Lambda \cap \Lambda'$) is a vertex lattice; in which case we have $\mathcal{Z}(\Lambda) \cap \mathcal{Z}(\Lambda') = \mathcal{Z}(\Lambda'')$ (resp. $\mathcal{Y}(\Lambda^\sharp) \cap \mathcal{Y}(\Lambda^{\prime \sharp}) = \mathcal{Y}(\Lambda^{\prime \prime \sharp})$).

    \item[(ii)] The intersection $\mathcal{Z}(\Lambda_0) \cap \mathcal{Z}(\Lambda_0')$ (or $\mathcal{Y}(\Lambda_0^\sharp) \cap \mathcal{Y}(\Lambda_0^{\prime \sharp})$) is always empty if $\Lambda_0 \ne \Lambda'_0$.
    
    \item[(iii)] The intersection $\mathcal{Z}(\Lambda) \cap \mathcal{Z}(\Lambda_0)$ (resp.  $\mathcal{Y}(\Lambda^\sharp) \cap \mathcal{Y}(\Lambda_0^\sharp)$) is non-empty if and only if $\Lambda \subset \Lambda_0$ (resp. $\Lambda_0 \subset \Lambda$), in which case $\mathcal{Z}(\Lambda) \cap \mathcal{Z}(\Lambda_0)$ (resp.  $\mathcal{Y}(\Lambda^\sharp) \cap \mathcal{Y}(\Lambda_0^\sharp)$) is isomorphic to $\PP_k^{t-h-1}\times \PP_k^{2h}$ (resp. $\PP_k^{h-t-1}\times \PP_k^{n-2h-1}$).
    
    \item[(iv)] The BT-strata $\mathcal{Z}(\Lambda_0)$ and $\mathcal{Y}(\Lambda_0^\sharp)$ are each isomorphic to the blow-up of $\PP_k^{n-1}$ along $\PP_k^{2h-1}$.
   \end{itemize}
\end{enumerate}
\end{Theorem}
The above BT-strata are representable by projective schemes over $\bar{k}$. By definition, $\mathcal{Z}(\Lambda_1)$ and $\mathcal{Y}(\Lambda_2^\sharp)$ are closed subschemes of the special fiber of $\calN^\spl_n$, and in Theorem~\ref{Reducedness} we prove that these subschemes are reduced. To establish reducedness, in \S\ref{LPBT} we show that these BT-strata are \'etale-locally isomorphic to certain simpler schemes---the strata models $\calM^{[2h]}_n(2t)$---which are closed subschemes of the special fiber of $\M_n^{\spl,[2h]}$ and are introduced in \S\ref{sec. 32}. In particular, we construct a local model diagram  
 \begin{equation}
\begin{tikzcd}
&\tilde{\mathcal{Z}}(\Lambda_1)\arrow[dl, "\psi_1"']\arrow[dr, "\psi_2"]  & \\
\mathcal{Z}(\Lambda_1)  &&  \calM^{[2h]}_n(2t)
\end{tikzcd}
\end{equation}
where $\calM^{[2h]}_n(2t)$ is the strata model with $t> h$.  The morphisms $\psi_1$ and $\psi_2$ are smooth of the same dimension. Similarly, we have a local model diagram for $\mathcal{Y}(\Lambda_2^\sharp)$ and $\calM^{[2h]}_n(2t)$ where $t< h$. 

Therefore, to obtain certain nice local properties for the BT-strata, it is enough to study $\calM^{[2h]}_n(2t)$. Similar to the splitting models associated to Shimura varieties, we explicitly calculate an open affine covering $\cup~ U_{i_0,j_0}$ of $\calM^{[2h]}_n(2t)$. Studying these affine schemes, in \S \ref{affinechart strata}, we deduce that: 

\begin{Theorem}\label{thm splstrataintro}
The strata model $\calM^{[2h]}_n(2t)$ is smooth. Moreover, 

(1). For $t> h$, the strata model $\calM^{[2h]}_n(2t)$ is of dimension $t+h$.

(2). For $t< h$, excluding the case where $n$ is even and $h=\frac n2$ ($\pi$-modular case), the strata model $\calM^{[2h]}_n(2t)$ is of dimension $n-t-h-1$.

(3). For $t_2<h<t_1$, the intersection of strata models $\calM^{[2h]}_n(2t_1)\cap \calM^{[2h]}_n(2t_2)$ is of dimension $t_1-t_2-1$.   
\end{Theorem}
We note that $\calM^{[2h]}_n(2t)$ admits a second description via the blow-up construction. Recall that $\M^{\spl,[2h]}_n$ is the blow-up of $\M^{\mathrm{loc},[2h]}_n$ at the worst point $\ast$, and let
\[
  \tau\colon \M^{\spl,[2h]}_n \longrightarrow \M^{\mathrm{loc},[2h]}_n
\]
denote the blow-up morphism. In \cite{HLS2} the strata local models $\M^{\mathrm{loc},[2h]}_n(2t)$ are constructed as certain closed subschemes of $\M^{\mathrm{loc},[2h]}_n$ attached to the corresponding BT-strata of the RZ space $\calN_n$ (in \emph{loc.\,cit.} these BT-strata are shown to be irreducible). By Proposition~\ref{Sec. 4.3}, $\calM^{[2h]}_n(2t)$ is the strict transform of $\M^{\mathrm{loc},[2h]}_n(2t)$ under $\tau$. As a byproduct of this description we obtain the irreducibility of the BT-strata $\mathcal{Z}(\Lambda_1)$ and $\mathcal{Y}(\Lambda_2^\sharp)$ (Theorem~\ref{Reducedness}).

We now explain the organization of the paper. In \S \ref{RZspaces} we review the relevant unitary RZ spaces and define the corresponding BT strata over their special fiber. In \S \ref{StrataBlModels} we introduce the strata models $\mathcal{M}^{[2h]}_n(2t)$ and, using an explicit affine cover, we prove Theorem \ref{thm splstrataintro}. Section \ref{BLupsection} recasts the construction via blow-ups: the splitting model is identified with the blow-up of the local model at the worst point $*$ and the strata models are described as strict transforms. In \S \ref{LPBT} we transfer these properties to the RZ space through the local-model diagram, obtaining smoothness, irreducibility, and dimension formulas for $\calZ(\Lambda)$, $\calY(\Lambda^{\sharp})$ and their intersections. Finally, \S \ref{BTstrat.} constructs the BT stratification of $\calN^\spl_{n,\mathrm{red}}$ and proves our main result, Theorem \ref{Thm BT intro}. The Appendix compares splitting-model variants and proves the identification $\mathscr{M}^{\bl,[2h]}_n \simeq M^{\mathrm{spl},[2h]}_n$.

\smallskip

{\bf Acknowledgements:} We thank Y. Luo for useful
suggestions. I.Z. was supported by Germany's Excellence Strategy EXC~2044--390685587 ``Mathematics M\"uns\-ter: Dynamics--Geometry--Structure'' and by the CRC~1442 ``Geometry: Deformations and Rigidity'' of the DFG.

\section{Rapoport-Zink spaces}\label{RZspaces}

In this section, we present the definition and basic properties of certain ramified unitary Rapoport–Zink (RZ) spaces, with level structure given by the stabilizer of a vertex lattice. As these spaces have already appeared in the literature, our discussion will be brief and accompanied by the relevant references.

\subsection{Preliminaries: Strict $O_{F_0}$-modules and hermitian $O_{F}$-modules}\label{sec2.1}
Let $F_0$ be a finite extension of $\mathbb{Q}_p$, where $p$ is an odd prime, with residue field $k = \mathbb{F}_q$. Let $\bar{k}$ be a fixed algebraic closure of $k$ and $F$ a ramified quadratic extension of $F_0$. Denote by $a \mapsto \bar{a}$ the (nontrivial) Galois involution of $F/F_0$ and let $\pi$ be a uniformizer of $F$ such that $\bar{\pi} = - \pi$. Let $ \pi_0 = \pi^2$, a uniformizer of $F_0$. Denote by $\br F$ the completion of the maximal unramified extension of $F$ and let $O_F$, $O_{\br F}$ be the ring of integers of $F$, $\br F$ respectively. Denote by $\mathrm{Nilp} \, O_{\br F}$ the category of $O_{\br F}$-schemes $S$ such that $\pi$ is locally nilpotent on $S$ and for such an $S$ denote its special fiber $S \times_{\mathrm{Spf}\,  O_{\br F}} \Spec \bar{k}$  by $\bar{S}$. Let $\sigma \in \mathrm{Gal}(\br{F}_0/F_0)$ be the Frobenius element. We fix an injection of rings $i_0 : O_{F_0} \rightarrow O_{\br{F}_0}$ and an injection $i : O_F \rightarrow O_{\br{F}}$ extending $i_0$. Denote by $ \bar{i} : O_F \rightarrow O_{\br{F}}$ the map $ a \mapsto i(\bar{a})$.

A 
strict $O_{F_0}$-module over $S$, where $S$ is an $O_{F_0}$-scheme, is a pair $(X,\iota)$ where $X$ is a $p$-divisible group over $S$ and $\iota : O_{F_0} \longrightarrow \operatorname{End}(X)$ is an action such that $O_{F_0}$ acts on $\operatorname{Lie}(X)$ via the structure morphism $O_{F_0} \to \mathcal{O}_S$. Such an $O_{F_0}$-module is called \emph{formal} if the underlying $p$-divisible group $X$ is formal. By Zink--Lau’s theory, which is generalized in \cite{ACZ}, there is an equivalence of categories between the strict formal $O_{F_0}$-modules over $S$ and nilpotent $O_{F_0}$-displays over $S$ (see also \cite[\S 3.1]{HLS2} and \cite[\S 5]{LMZ} for more details).
 To any strict formal $O_{F_0}$-module, there is an associated crystal $\mathbb{D}_X$ on the category of $O_{F_0}$-pd-thickenings. We define the (covariant relative) de~Rham realization as $D(X) := \mathbb{D}_X(S)$ and by the (relative) Grothendieck--Messing theory we obtain a short exact sequence of $\mathcal{O}_S$-modules:
\[
0 \longrightarrow \operatorname{Fil}(X) \longrightarrow D(X) \longrightarrow \operatorname{Lie}(X) \longrightarrow 0,
\]
where $\operatorname{Fil}(X) \subset D(X)$ is the Hodge filtration. (See \cite[\S 3.1]{HLS2} for a more comprehensive treatment.)

Next, we restrict to the case where $X = (X, \iota)$ is biformal; see \cite[Definition~11.9]{Mih} for the definition. For a biformal strict $O_{F_0}$-module $X$, we can define the (relative) dual $X^\vee$ of $X$, and hence the (relative) polarization and the (relative) height. From the definition, it follows that there is a perfect pairing 
\begin{equation}\label{perf.pair.}
D(X) \times D(X^\vee) \to \mathcal{O}_S
\end{equation}
such that $\operatorname{Fil}(X) \subset D(X)$ and $\operatorname{Fil}(X^\vee) \subset D(X^\vee)$ are orthogonal complements of each other and there are two induced perfect pairings 
\begin{equation}\label{perfectpair}
\operatorname{Fil} (X) \times \operatorname{Lie}(X^\vee)\to \mathcal{O}_S \text{  and  } \operatorname{Fil} (X^\vee) \times \operatorname{Lie}(X)\to \mathcal{O}_S.    
\end{equation}

When $S = \operatorname{Spec} R$ is perfect, the nilpotent $O_{F_0}$-display is equivalent to the relative Dieudonn\'e module $M(X)$ over $W_{O_{F_0}}(R)$, equipped with a $\sigma$-linear operator $F$ and a $\sigma^{-1}$-linear operator $V$, such that $FV = VF = \pi \cdot \mathrm{id}$. (Here, $W_{O_{F_0}}(R)$ is the ring of ramified Witt vectors and we refer the reader to \cite[\S 3.1]{HLS2} for more details.)
 
\begin{Definition}\label{def 21}
\rm{
 Let $h, n$ be integers with $0 \leq h \leq \lfloor \frac n2 \rfloor$. For any $S \in \mathrm{Nilp}\, O_{\br F}$, a \textit{hermitian} $O_F$-\textit{module} of rank $n$ and type $2h$ (with signature $(n-1, 1)$) over $S$ is a triple $(X, \iota, \lambda)$ satisfying:
\begin{enumerate}
  \item $X$ is a strict biformal $O_{F_0}$-module over $S$ of height $2n$ and dimension $n$.
  \item $\iota : O_F \to \operatorname{End}(X)$ is an action of $O_F$ on $X$ extending the $O_{F_0}$-action.
  \item $\lambda$ is a (relative) polarization of $X$ that is $O_F/O_{F_0}$-semilinear in the sense that the Rosati involution $\operatorname{Ros}_\lambda$ induces the non-trivial involution $\sigma \in \operatorname{Gal}(F/F_0)$ on $\iota : \mathcal{O}_F \to \operatorname{End}(X)$.
  \item We require that $\ker[\lambda] \subset X[\iota(\pi)]$ and has order $q^{2h}$.
\end{enumerate}
}
\end{Definition}
From (4) above, we deduce that there exists a unique isogeny $\lambda^\vee : X^\vee \to X$ such that $\lambda \circ \lambda^\vee = \iota(\pi)$ and $\lambda^\vee \circ \lambda = \iota(\pi)$.

\subsection{Unitary RZ spaces}\label{sec2.2} We fix a supersingular hermitian $O_F$-module $(\mathbb{X},\iota_{\mathbb{X}},\lambda_{\mathbb{X}})$ over $\Spec \bar{k}$ of rank $n$ and type $2h$ (with signature $(n-1, 1)$) which we call the framing object; supersingular means that the rational  Dieudonn\'e module $N = M(\mathbb{X})[1/\pi_0]$ has all relative slopes $\frac{1}{2}$. (We refer to \cite[\S 5]{LRZ} for the existence of these framing objects.) Now, we are ready to define the following RZ spaces which are \textit{relative} in the sense of \cite{Mih}. 

\begin{Definition}
{\rm 
\begin{enumerate}
    \item The wedge RZ space $\mathcal{N}_n^{\wedge}$ is the set-valued functor on $\mathrm{Nilp} \, O_{\breve{F}}$ which associates to $S \in \mathrm{Nilp} \, O_{\breve{F}}$ the set of isomorphism classes of quadruples $(X, \iota, \lambda, \rho)$ which satisfy
\begin{enumerate}
    \item  $(X, \iota, \lambda)$ is a hermitian  $O_F$-module over S of dimension $n$ and type $2h$.
    \item $\rho : X \times_S \bar{S} \to \mathbb{X} \times_{\bar{k}} \bar{S}$ is an $O_F$-linear quasi-isogeny of height $0$ over the special fiber $\bar{S} = S \times_{\mathrm{Spf}\,  O_{\br F}} \Spec \bar{k}$ such that $\rho^*(\lambda_{\mathbb{X}, \bar{S}}) = \lambda_{\bar{S}}$.
    \item The action of $O_F$ on $\operatorname{Fil}(X)$ induced by $\iota : O_F \to \operatorname{End}(X)$ satisfies:
 \begin{itemize}
  \item \textit{(Kottwitz condition)}: 
  $ \operatorname{char}(\iota(\pi) \mid \operatorname{Fil}(X)) = (T - \pi)(T + \pi)^{n-1}.$   \item \textit{(Wedge condition)}: 
 $\wedge^2(\iota(\pi) - \pi \mid \operatorname{Fil}(X)) = 0, \,  \wedge^n(\iota(\pi) + \pi \mid \operatorname{Fil}(X)) = 0.$ 
 \end{itemize}
 \item \textit{(Spin condition)} When $n$ is even and $2h = n$, we ask that $\iota(\pi) - \pi$ is non-vanishing on $\mathrm{Fil}(X)$.
\end{enumerate}

\item The RZ space $\mathcal{N}_n$ is defined as the closed formal subscheme of $\mathcal{N}_n^{\wedge}$ cut out by the ideal sheaf $\mathcal{O}_{\mathcal{N}_n^{\wedge}}[\pi_0^\infty] \subset \mathcal{O}_{\mathcal{N}_n^{\wedge}}$. This is the maximal flat closed formal subscheme of $\mathcal{N}_n^{\wedge}$.
\end{enumerate}
}\end{Definition}

The RZ spaces $\mathcal{N}_n^{\wedge}$ and $\mathcal{N}_n$ are representable by formal schemes locally of finite type over $\mathrm{Spf} \,O_{\br F}$ and both spaces have relative dimension $n - 1$ (see \cite[\S 3.3]{HLS2}). The closed formal subscheme $\mathcal{N}_n$ is flat and has the same underlying topological space with $\mathcal{N}_n^{\wedge}$, i.e. these spaces share identical reduced loci. These assertions can be easily seen by using the local model diagram and passing to the corresponding {\it wedge local model} ${\rm M}_n^{\wedge}$ and the {\it local model} ${\rm M}_n^{\rm loc}$ (see \cite[Proposition 3.4]{HLS2}). From \cite[\S 3.3]{HLS2}, we also see that $\mathcal{N}_n$ is a linear modification of $\mathcal{N}_n^{\wedge}$ in the sense of \cite[\S 2]{P}.
\begin{Definition}\label{nRegDef}
\rm{
 The {\it Kr\"amer RZ space} $\mathcal{N}_n^{\rm Kra}$ (of type $2h$, height $n$) is the set-valued functor on $\mathrm{Nilp} \, O_{\breve{F}}$, sending $S \in \mathrm{Nilp} \, O_{\breve{F}}$ to the set of isomorphism classes of tuples $(X, \iota, \lambda, \rho,\operatorname{Fil}^0(X), \operatorname{Fil}^0(X^\vee))$ which satisfy
\begin{enumerate}
 \item  $(X, \iota, \lambda)$ is a hermitian  $O_F$-module over $S$ of dimension $n$ and type $2h$.
    \item $\rho : X \times_S \bar{S} \to \mathbb{X} \times_{\bar{k}} \bar{S}$ is an $O_F$-linear quasi-isogeny of height $0$ over the special fiber $\bar{S}$ such that $\rho^*(\lambda_{\mathbb{X}, \bar{S}}) = \lambda_{\bar{S}}$.
    \item $\operatorname{Fil}^0(X)$ (resp. $\operatorname{Fil}^0(X^\vee)$) is locally a $\mathcal{O}_S$-direct summand of the Hodge filtration $\operatorname{Fil}(X) \subset D(X)$ (resp. $\operatorname{Fil}(X^\vee) \subset D(X^\vee)$) of rank one that satisfies the \textit{transition conditions}
    \[
\lambda (\operatorname{Fil}^0(X)) \subset \operatorname{Fil}^0(X^\vee), \quad (\text{resp.} \quad \lambda^\vee (\operatorname{Fil}^0(X^\vee))\subset \operatorname{Fil}^0(X))
     \]
and \textit{splitting conditions}
    \[  (\iota(\pi)+\pi) (\operatorname{Fil}(X)) \subset \operatorname{Fil}^0(X), \quad  (\iota(\pi)-\pi) (\operatorname{Fil}^0(X))=0.
    \]  
    \[ (\text{resp}.\quad  (\iota(\pi)+\pi) (\operatorname{Fil}(X^\vee)) \subset \operatorname{Fil}^0(X^\vee), \quad  (\iota(\pi)-\pi) (\operatorname{Fil}^0(X^\vee))=0.)
    \]
\end{enumerate}
}
\end{Definition}
The Kr\"amer RZ space $\mathcal{N}_n^{\rm Kra}$ is represented by a formal scheme, locally of finite type over $\mathrm{Spf}\, O_{\br F}$, of relative dimension $n-1$; representability follows from the general results of \cite{RZbook}. From \cite[\S5.2]{HLS1}, we obtain the local model diagram 
 \begin{equation}\label{LMdiagramKra}
\begin{tikzcd}
&\widetilde{\mathcal{N}}_n^{\rm Kra}\arrow[dl, "\pi"']\arrow[dr, "\tilde{\varphi}"]  & \\
\mathcal{N}_n^{\rm Kra} &&  {\rm M}_n^{\rm nspl}
\end{tikzcd}
\end{equation}
where ${\rm M}_n^{\rm nspl}$ is the {\it naive splitting model} and the morphisms $\tilde{\varphi}$ and $\pi$ are smooth of equal
relative dimension. The naive splitting model admits an explicit moduli description and is equipped with a projective forgetful morphism $\tau:\ {\rm M}_n^{\rm nspl}\longrightarrow {\rm M}_n^{\rm loc}$ whose restriction to the generic fiber is an isomorphism (\cite[\S2.3]{HLS1}). We refer to this model as naive since ${\rm M}_n^{\rm nspl}$, and so $\mathcal{N}_n^{\rm Kra} $, is not in general flat. To remedy this non-flatness, one can either try to modify the moduli problem and obtain a flat model or define the splitting model as the scheme theoretic closure of the generic fiber of ${\rm M}_n^{\rm nspl}$ in ${\rm M}_n^{\rm nspl}$. 

The first route was taken in \cite{ZacZhao2} and the corresponding RZ space is defined analogously to $\mathcal{N}_n^{\rm Kra}$ in Definition \ref{nRegDef}, except in (3), where we impose only the additional filtration $\operatorname{Fil}^0(X)$ inside $\operatorname{Fil}(X)$ (omitting $\operatorname{Fil}^0(X^\vee)$) together with the corresponding splitting conditions. This RZ space is flat and normal and its Bruhat-Tits stratification was studied in \cite{ZacZhao3}. 

The latter route was taken in \cite{HLS1}, where the authors defined ${\rm M}_n^{\rm spl}$ to be the {\it splitting model} that is defined as the scheme-theoretic closure of the generic fiber ${\rm M}_n^{\rm nspl}\otimes_{O_F}F$ in ${\rm M}_n^{\rm nspl}$.

\begin{Definition}\label{RegDef}
{\rm
We define the flat closure $\mathcal{N}^\spl_n \subset \mathcal{N}_n^{\rm Kra} $ to be the closed formal subscheme defined by the ideal sheaf $\mathcal{O}_{\mathcal{N}_n^{\rm Kra}}[\pi_0^\infty] \subset \mathcal{O}_{\mathcal{N}_n^{\rm Kra}}$. 
}\end{Definition}

As in the local-model case, the RZ space $\mathcal{N}^\spl_n$ is a linear modification of $\mathcal{N}_n^{\rm Kra}$, and there is a local-model diagram connecting $\mathcal{N}^\spl_n$ to the splitting model ${\rm M}_n^{\rm spl}\subset {\rm M}_n^{\rm nspl}$. By \cite[Theorem 1.3.1]{HLS1}, ${\rm M}_n^{\rm spl}$ is flat with semi-stable reduction; hence the same holds for $\mathcal{N}^\spl_n$. In Section \ref{StrataBlModels}, we will give the explicit definitions of the (naive) splitting models and discuss the relations among all the above different variants.


\begin{Remark}\label{DualSpace}
{\rm
Denote by $\mathcal{F}_X \subset  \operatorname{Lie}(X)$ (resp. $\mathcal{F}_{X^\vee} \subset  \operatorname{Lie}(X^\vee)$) the perpendicular complement of $ \operatorname{Fil}^0(X)$ (resp. $ \operatorname{Fil}^0(X^\vee)$) under the perfect pairing (\ref{perfectpair}) which determine each other. Then, as in \cite[\S 3.2]{HLSY}, the splitting condition of Definition \ref{nRegDef} is equivalent to: $\mathcal{F}_X$ (resp $\mathcal{F}_{X^\vee}$) is a local direct summand of $\operatorname{Lie} (X)$ (resp. $\operatorname{Lie} (X^\vee)$) of rank $n - 1$ as an $\mathcal{O}_S$-module such that $\mathcal{O}_F$ acts on $\mathcal{F}_X$ (resp. $\mathcal{F}_{X^\vee}$) via $O_F \xrightarrow{i} O_{\br{F}} \to \mathcal{O}_S$ and acts on $\operatorname{Lie} (X) / \mathcal{F}_X$ (resp. $\operatorname{Lie} (X^\vee) / \mathcal{F}_{X^\vee}$) via $O_F \xrightarrow{\bar{i}} O_{\br{F}} \to \mathcal{O}_S$. 
}\end{Remark}

Recall that we denote by $N = M(\mathbb{X})[1/\pi_0]$ the rational Dieudonn\'e module of the framing object which is a $2n$-dimensional $\breve{F}_0$-vector space equipped with a $\sigma$-linear operator $F$ and a $\sigma^{-1}$-linear operator $V$. 
The $O_F$-action $\iota_{\mathbb{X}} : O_F \to \operatorname{End}(\mathbb{X})$ induces on $N$ an $O_F$-action commuting with $F$ and $V$. We still denote this induced action by $\iota_{\mathbb{X}}$ and denote $\iota_{\mathbb{X}}(\pi)$ by $\Pi$.

The polarization of $\mathbb{X}$ induces a skew-symmetric $\breve{F}_0$-bilinear form $\langle \cdot, \cdot \rangle$ on $N$ satisfying
\[
\langle Fx, y \rangle = \langle x, Vy \rangle^\sigma, \qquad \langle \iota(a)x, y \rangle = \langle x, \iota(\bar{a})y \rangle,
\]
for any $x, y \in N$, $a \in O_F$. Also, $N$ is an $n$-dimensional $\breve{F}$-vector space equipped with the $F/\breve{F}_0$-hermitian form $h(\cdot, \cdot)$ defined by
\[
h(x, y) := \delta \left( \langle \Pi x, y \rangle + \pi \langle x, y \rangle \right),
\]
where $\delta$ is a fixed element in $\breve{F}_0^\times$ satisfying $\sigma(\delta) = -\delta$. The bilinear form $\langle \cdot, \cdot \rangle$ can be recovered from $h(\cdot, \cdot)$ via the relation:
\[
\langle x, y \rangle = \frac{1}{2\delta} \operatorname{Tr}_{F/\breve{F}_0} \left( \pi^{-1} h(x, y) \right).
\]
For a lattice $\Lambda \subset N$ we denote by $ \Lambda^\sharp = \{x \in N \,|\, h(x,\Lambda) \in O_F \}$ its hermitian dual. Let $\tau := \Pi V^{-1}$ and define $C := N^{\tau = 1}$ (the set of $\tau$-fixed points in $N$). Then $C$ is an $F$-vector space of dimension $n$ and we have
\[
N = C \otimes_{F_0} \breve{F}_0.
\]
The $F/F_0$-hermitian form $h(\cdot, \cdot)$ restricts to $C$ and we continue to use the same notation for the restricted form. From now on, we write $\pi$ instead of $\Pi$ for the action on $C$. We call $\Lambda\subset C$ a vertex lattice if it satisfies 
\begin{equation}
\pi \Lambda^{\sharp} \subset \Lambda \subset  \Lambda^{\sharp}.   
\end{equation}
We denote by $t(\Lambda):=\dim(\Lambda^\sharp/\Lambda)$ the type of a vertex lattice, which is an even integer (see \cite[Lemma 3.2]{RTW}). Also, set $\breve{\Lambda} = \Lambda \otimes_{O_F} O_{\breve{F}}$. 

\begin{Proposition}\label{DLoc}
Let $ \kappa$ be a perfect field over $ \bar{k}$. There is a bijection between $\mathcal{N}_n(\kappa)$ and the set of $W_{O_{F_0}}(\kappa)$-lattices
{\small
\[
\left\{ 
M \subset N \otimes W_{O_{F_0}}(\kappa) \,\middle|\, 
\Pi M^\sharp \subset M \mathrel{\overset{\scriptstyle 2h}{\subset}} M^\sharp,\ 
\Pi M \subset \tau^{-1}(M) \subset \Pi^{-1} M,\ 
M \mathrel{\overset{\scriptstyle \leq 1}{\subset}} M + \tau(M) 
\right\}.
\]
}
\end{Proposition}
\begin{proof}
See \cite[Proposition 3.5]{HLS2}. 
\end{proof}

\begin{Proposition}\label{DSpl}
Let $ \kappa$ be a perfect field over $ \bar{k}$. There is a bijection between $\mathcal{N}_n^{\rm Kra}(\kappa)$ and the set of triples of $W_{O_{F_0}}(\kappa)$-lattices $(M,M', M'')$ in $N \otimes W_{O_{F_0}}(\kappa)$ satisfying
\[
\begin{array}{l}
\Pi M^\sharp \subset M \mathrel{\overset{\scriptscriptstyle 2h}{\subset}} M^\sharp,\quad 
\Pi M \subset \tau^{-1}(M) \subset \Pi^{-1} M,\quad   \Pi M'\subset M''\subset M'\\
V M^\sharp \subset M' \subset \tau^{-1}(M^\sharp) \cap M^\sharp,\quad
\operatorname{length}(M^\sharp / M') = 1.\\   
V M \subset M'' \subset \tau^{-1}(M) \cap M,\quad
\operatorname{length}(M / M'') = 1.
\end{array}
\]
\end{Proposition}
\begin{proof}
Let $(X, \iota, \lambda, \rho,\operatorname{Fil}^0(X), \operatorname{Fil}^0(X^\vee)) \in \mathcal{N}_n^{\rm Kra}(\kappa) $ and let $M(X)$ be the $O_{F_0}$-relative Dieudonné module of $X$. Define $M = \rho(M(X)) \subset N  \otimes W_{O_{F_0}}(\kappa)$ and 
\[
M' = \rho(\operatorname{Pr}_1^{-1}(\mathcal{F}_{X^\vee})) \quad (\text{resp.}\quad
M'' = \rho(\operatorname{Pr}_2^{-1}(\mathcal{F}_X)))
\]
where $\operatorname{Pr}_1 : M(X^\vee) \to \operatorname{Lie} (X^\vee) = M(X^\vee)/V M(X^\vee)$ (resp. $\operatorname{Pr}_2 : M(X) \to \operatorname{Lie} (X) = M(X)/V M(X)$) is the natural quotient map.

As in the proof of \cite[Proposition 3.5]{HLS2} the relation $\Pi M^\sharp \subset M \mathrel{\overset{\scriptscriptstyle 2h}{\subset}} M^\sharp$ comes from the polarization $\lambda$ and the relation $\Pi M \subset \tau^{-1}(M) \subset \Pi^{-1} M$ is equivalent to $ \pi_0 M \subset VM \subset M $. (Note that the Hodge filtration $\operatorname{Fil}(X) \subset D(X)$ can be identified with $ VM / \pi_0M \subset M / \pi_0M$.) 
The transition conditions are equivalent to $\Pi M'\subset M''\subset M'$. Finally, the conditions $V M^\sharp \subset M' \subset \tau^{-1}(M^\sharp) \cap M^\sharp$ and $\operatorname{length}(M^\sharp/M') = 1$ translate to
\[V M^\sharp \subset M' \subset M^\sharp, \quad \Pi M' \subset V M^\sharp, \quad \dim_{\kappa}(M^\sharp/M') = 1,
\] 
which are in turn equivalent to 
\[
\mathcal{F}_{X^\vee} \subset \operatorname{Lie} (X^\vee), \quad \dim_{\kappa}(\mathcal{F}_{X^\vee}) = n - 1, \quad \Pi \cdot \mathcal{F}_{X^\vee} = \{0\}, \quad \Pi \cdot \operatorname{Lie} (X^\vee) \subset \mathcal{F}_{X^\vee}.
\]
Similarly, $V M \subset M'' \subset \tau^{-1}(M) \cap M$ and $\operatorname{length}(M / M'') = 1$ translate to 
\[
\mathcal{F}_X \subset \operatorname{Lie} X, \quad \dim_{\kappa}(\mathcal{F}_X) = n - 1, \quad \Pi \cdot \mathcal{F}_X = \{0\}, \quad \Pi \cdot \operatorname{Lie} (X) \subset \mathcal{F}_X.
\]
Hence the filtration $\mathcal{F}_X \subset \operatorname{Lie} (X)$ (resp. $\mathcal{F}_{X^\vee} \subset \operatorname{Lie} (X^\vee) $) satisfies the splitting conditions by Remark \ref{DualSpace}. We have translated all conditions in the definition of $\mathcal{N}_n^{\rm Kra}$ in terms of relative Dieudonné modules.
\end{proof}

\subsection{Bruhat-Tits strata}\label{BT_strata} 
We fix a one-dimensional hermitian $O_F$-module $ (\mathbb{Y}, \iota_{\mathbb{Y}}, \lambda_{\mathbb{Y}})$ over $\Spec \bar{k}$ of type 0. Define
\begin{equation} \label{eq:Vn-def}
\mathbb{V} := \operatorname{Hom}_{O_F}(\mathbb{Y}, \mathbb{X}) \otimes_{O_F} F.
\end{equation}
The vector space $\mathbb{V} $ is equipped with a hermitian form $(\, ,\, )_{\mathbb{V}} $ such that for any $x, y \in \mathbb{V} $,
\begin{equation} \label{eq:hermitian-V}
(x, y)_{\mathbb{V} } = \lambda_{\mathbb{Y}}^{-1} \circ y^\vee \circ \lambda_{\mathbb{X}} \circ x \in \operatorname{End}(\mathbb{Y}) \otimes_{O_F} F \cong F, 
\end{equation}
where $y^\vee$ is the dual quasi-homomorphism of $y$ and the last isomorphism is given by $\iota^{-1}_{\mathbb{Y}}$. The hermitian spaces $(\mathbb{V}, (\, ,\, )_{\mathbb{V}})$ and $(C, h(\,, \,))$ are related by the $F$-linear isomorphism
\begin{equation} \label{eq:b-isom}
b : \mathbb{V} \to C, \quad x \mapsto x(e),
\end{equation}
where $e$ is a generator of the $\tau$-fixed points of the $O_{F_0}$-relative Dieudonné module $M(\mathbb{Y})$; in particular, $\mathbb{V}$ and $C$ are isomorphic as hermitian spaces (see \cite[\S 2.2]{HLSY}). We will sometimes identify $\mathbb{V}$ with $C$.
\quash{
\begin{Definition}
  \rm{
(1) For an $O_F$ lattice $L \subset \mathbb{V}$, define the subfunctor $\mathcal{Z}(L)$ of $ \mathcal{N}_n$ such that $\mathcal{N}_n(S)$ is the set of isomorphism classes of tuples $ (X, \iota, \lambda, \rho) \in \mathcal{N}_n(S)$ such that for any $x \in L \subset \mathbb{V}$ the quasi-homomorphism $ \rho^{-1} \circ x \circ \rho_Y: \mathbb{Y} \times_{\bar{k}} \bar{S} \to X \times_S \bar{S}$ extends to a homomorphism $ Y \to X $.

(2) For an $O_F$ lattice $L \subset \mathbb{V}$, define the subfunctor $\mathcal{Y}(L^{\sharp})$ of $ \mathcal{N}_n$ such that $\mathcal{N}_n^\loc(S)$ is the set of isomorphism classes of tuples $ (X, \iota, \lambda, \rho) \in \mathcal{N}_n(S)$ such that for any $x^{^\sharp} \in \Lambda^{\sharp} \subset \mathbb{V}$ the quasi-homomorphism $ \rho^\vee \circ \lambda_{\mathbb{X}} \circ x^{^\sharp} \circ \rho_Y: \mathbb{Y} \times_{\bar{k}} \bar{S} \to X^\vee \times_S \bar{S}$ extends to a homomorphism $ Y \to X^\vee $.
}
\end{Definition}
}

By (relative) Dieudonn\'e theory, the lattices $\breve{\Lambda}$ and $\breve{\Lambda}^\sharp$ correspond to the strict $O_{\breve{F}_0}$-modules $X_\Lambda$ and $X_{\Lambda^\sharp}$ over $\bar{k}$, respectively, with quasi-isogenies $\rho_\Lambda : X_\Lambda \to \mathbb{X}$ and $\rho_{\Lambda^\sharp} : X_{\Lambda^\sharp} \to \mathbb{X}$. 
We define the following two kinds of Bruhat--Tits (BT) strata for the special fiber $\overline{\mathcal{N}}_n$  of $\mathcal{N}_n$ (see also \cite[Definition 2.2]{HLS2}):

\begin{Definition}\label{Isogenies}
\rm{
Fix an even integer $0\leq 2h\leq n$. Let $L_{\mathcal{Z}}$ denote the set of all vertex lattices in $C$ of type $2t \geq 2h$, and let $L_{\mathcal{Y}}$ denote the set of all vertex lattices in $C$ of type $2t\leq 2h$.
\begin{enumerate}
  \item For any $\Lambda \in L_{\mathcal{Z}}$, the $\mathcal{Z}$-stratum $\mathcal{Z}^{\rm loc}(\Lambda)$ is the subfunctor of $\overline{\mathcal{N}}_n$ that assigns to each $\bar{k}$-scheme $S$ the set of tuples $(X, \iota, \lambda, \rho)$ such that the composition $\rho_{\Lambda,X} := \rho^{-1} \circ (\rho_\Lambda)_S$ is an isogeny.

  \item For any $\Lambda \in L_{\mathcal{Y}}$, the $\mathcal{Y}$-stratum $\mathcal{Y}^{\rm loc}(\Lambda^\sharp)$ is the subfunctor of $\overline{\mathcal{N}}_n$ that assigns to each $\bar{k}$-scheme $S$ the set of tuples $(X, \iota, \lambda, \rho)$ such that the composition $\rho_{\Lambda^\sharp, X^\vee} := \rho^\vee \circ \lambda_{\mathbb{X}} \circ \rho_{\Lambda^\sharp}$ is an isogeny, where $\rho_{\Lambda^\sharp} = \rho_\Lambda \circ \lambda_\Lambda^{-1}$.
\end{enumerate}
}
\end{Definition}
By \cite[Lemma 2.10]{RZbook}, $\mathcal{Z}^{\rm loc}(\Lambda)$ and $\mathcal{Y}^{\rm loc}(\Lambda^\sharp)$ are closed formal subschemes of $\overline{\mathcal{N}}_n$ (see also \cite[\S 2.1]{HLS2}). Using the same reasoning as in \cite[Lemma 4.2]{VW}, it follows that they are representable by projective schemes over $\bar{k}$. Also, these schemes are reduced (see \cite[Corollary 4.8]{HLS2}) and so they lie in the reduced subscheme $  \calN_{n, {\rm red}}$ of $  \overline{\mathcal{N}}_n$.

\begin{Definition}\label{def 2.10}
{\rm
With notation as above,  we define the following two kinds of BT strata for the special fiber of the Kr\"amer RZ-space $\overline{\mathcal{N}}_n^{\rm Kra}$.
\begin{enumerate}
    \item $\mathcal{Z}^{\rm Kra}(\Lambda) (S) $ is the set of isomorphism classes of tuples $(X, \iota, \lambda, \rho, \mathcal{F}_X, \mathcal{F}_{X^\vee}) \in \overline{\mathcal{N}}_n^{\rm Kra} (S)$ such that $(X, \iota, \lambda, \rho) \in \mathcal{Z}^{\rm loc}(\Lambda)(S)$ and if $\Lambda$ is of type $2t \neq 2h $, we require in addition that $ x_{*} (\operatorname{Lie}(Y\times S)) \subset \mathcal{F}_X \cap\calF_{X^\vee}$ for any $x \in \Lambda$. 
    \item $\mathcal{Y}^{\rm Kra}(\Lambda^\sharp) (S)$ is the set of isomorphism classes of tuples $(X, \iota, \lambda, \rho, \mathcal{F}) \in \overline{\mathcal{N}}_n^{\rm Kra}(S) $ such that $(X, \iota, \lambda, \rho) \in \mathcal{Y}^{\rm loc}(\Lambda^\sharp)(S)$ and if $\Lambda$ is of type $2t \neq 2h $, we require in addition that $ x^{\sharp}_{*} (\operatorname{Lie}(Y\times S)) \subset \mathcal{F}_X \cap\calF_{X^\vee}$ for any $x^{\sharp} \in \Lambda^{\sharp}$. 
\end{enumerate}
}\end{Definition}

Recall that we have the forgetful map $\phi: \overline{\mathcal{N}}^\spl_n \rightarrow \overline{\mathcal{N}}_n$. The corresponding strata for the special fiber of the maximal flat closed formal subscheme $ \overline{\mathcal{N}}^\spl_n$ are defined as follows:
\begin{equation}\label{eq ZYstrata}
\mathcal{Z}(\Lambda) : = \mathcal{Z}^{\rm loc}(\Lambda) \times_{ \overline{\mathcal{N}}_n}  \overline{\mathcal{N}}^\spl_n \quad \text{and} \quad \mathcal{Y}(\Lambda^\sharp) : = \mathcal{Y}^{\rm loc}(\Lambda^\sharp) \times_{ \overline{\mathcal{N}}_n}  \overline{\mathcal{N}}^\spl_n,  
\end{equation}
which are closed formal subschemes of $\overline{\mathcal{N}}^\spl_n$. 

Using Propositions \ref{DLoc} and \ref{DSpl} we can naturally obtain a lattice-theoretic characterization of the BT-strata of $\overline{\mathcal{N}}_n$ and $\overline{\mathcal{N}}_n^{\rm Kra}$:
\begin{Proposition}\label{DieudLatt}
Let $ \kappa$ be a perfect field over $\bar{k}$. The $\kappa$-points of the BT-strata can be described as follows: 

(1) Assume $\Lambda\subset C$ is a vertex lattice of type $2t \geq 2h$.
\begin{itemize}
    \item For $t= h$, we have
    \begin{flalign*}
    & \mathcal{Z}^{\rm loc}(\Lambda)(\kappa) = \left\{ (X, \iota, \lambda, \rho) \in \overline{\mathcal{N}}_n(\kappa) \mid \Lambda \otimes W_{O_{F_0}}(\kappa) = M(X)  \right\}, \\  
    & \mathcal{Z}^{\rm Kra}(\Lambda)(\kappa) = \left\{ (X, \iota, \lambda, \rho, \mathcal{F}_X, \mathcal{F}_{X^\vee}) \in \overline{\mathcal{N}}^{\rm Kra}_{n}(\kappa) \mid \Lambda \otimes W_{O_{F_0}}(\kappa) = M(X)  \right\}.
    \end{flalign*}
  \item For $t > h$, we have that $ \mathcal{Z}^{\mathrm{loc}}(\Lambda)(\kappa)$ is equal to
   \[  
  \left\{ (X, \iota, \lambda, \rho) \in \overline{\mathcal{N}}_n(\kappa) \;\middle|\;
  \Lambda \otimes W_{O_{F_0}}(\kappa) \subset M(X) \subset M(X)^\sharp \subset \Lambda^\sharp \otimes W_{O_{F_0}}(\kappa)
  \right\}
   \]
and $\mathcal{Z}^{\mathrm{Kra}}(\Lambda)(\kappa) $ is equal to 
\[
\left\{ (X, \iota, \lambda, \rho, \mathcal{F}_X, \mathcal{F}_{X^\vee}) \in \overline{\mathcal{N}}^{\rm Kra}_{n}(\kappa) ~\middle|~
  \begin{aligned}
    &(X, \iota, \lambda, \rho) \in \mathcal{Z}^{\mathrm{loc}}(\Lambda)(\kappa), \\
    &\Lambda \otimes W_{O_{F_0}}(\kappa) \subset M'(X) \subset M(X)^\sharp,\\
    &\Lambda \otimes W_{O_{F_0}}(\kappa) \subset M''(X) \subset M(X)
  \end{aligned}
  \right\}.
\]
\end{itemize}

(2) Assume $\Lambda\subset C$ is a vertex lattice of type $2t \leq 2h$.
\begin{itemize}
    \item For $t=h$, we have
  \begin{flalign*}
   &  \mathcal{Y}^{\rm loc}(\Lambda^{\sharp})(\kappa) = \left\{ (X, \iota, \lambda, \rho) \in \overline{\mathcal{N}}_n(\kappa) \mid \Lambda \otimes W_{O_{F_0}}(\kappa) = M(X)  \right\}, \\  
  &  \mathcal{Y}^{\rm Kra}(\Lambda^{\sharp})(\kappa) = \left\{ (X, \iota, \lambda, \rho, \mathcal{F}_X, \mathcal{F}_{X^\vee}) \in \overline{\mathcal{N}}^{\rm Kra}_{n}(\kappa) \mid \Lambda \otimes W_{O_{F_0}}(\kappa) = M(X)  \right\}.
    \end{flalign*}

\item For $t < h$, we have that $ \mathcal{Y}^{\mathrm{loc}}(\Lambda^\sharp)(\kappa)$ is equal to
\[
\left\{ (X, \iota, \lambda, \rho) \in \overline{\mathcal{N}}_n(\kappa) \;\middle|\;
  \pi \Lambda^\sharp \otimes W_{O_{F_0}}(\kappa) \subset \pi M(X)^\sharp \subset M(X) \subset \Lambda \otimes W_{O_{F_0}}(\kappa)
  \right\}
\]
and $\mathcal{Y}^{\mathrm{Kra}}(\Lambda^\sharp)(\kappa) $ is equal to 
\[
\left\{ (X, \iota, \lambda, \rho, \mathcal{F}_X, \mathcal{F}_{X^\vee}) \in \overline{\mathcal{N}}^{\rm Kra}_{n}(\kappa) ~\middle|~
  \begin{aligned}
    &(X, \iota, \lambda, \rho) \in \mathcal{Z}^{\mathrm{loc}}(\Lambda)(\kappa), \\
    &\Lambda^\sharp \otimes W_{O_{F_0}}(\kappa) \subset M'(X) \subset M(X)^\sharp,\\
    &\Lambda^\sharp \otimes W_{O_{F_0}}(\kappa) \subset M''(X) \subset M(X)
  \end{aligned}
  \right\}.
\]
\end{itemize}
\qed
\end{Proposition}

Let $ \kappa$ be a perfect field over $\bar{k}$ and let $\Lambda_0\subset C$ be a vertex lattice of type $2h$. By \cite[Corollary 2.11]{ZacZhao3}, we have that $ \mathcal{Z}^{\rm loc}(\Lambda_0)(\kappa)$ is equal to $  \mathcal{Y}^{\rm loc}(\Lambda_0^{\sharp})(\kappa)$. Both of them as sets contain a discrete point $\{\Lambda_0 \otimes W_{O_{F_0 }}(\kappa) \}$ (called the {\it worst point}) in the RZ space. Note that under the local model diagram, the points $ M = \Lambda_0 \otimes W_{O_{F_0 }}(\kappa) \in \mathcal{Z}^{\rm loc}(\Lambda_0)(\kappa)$, where $\Lambda_0$ is a vertex lattice of type $2h$, correspond to the worst point of the associated local model ${\rm M}_n^{\rm loc}$ (see \cite[\S 5.3]{LRZ}).

\section{Strata models}\label{StrataBlModels}
In this section, we introduce the strata models attached to the BT strata of the special fibers $\overline{\calN}_n^\Kra$ and $\overline{\calN}_n^\spl$. We begin by recalling the (naive) splitting models. For each lattice $\Lambda$ (resp.\ $\Lambda^\sharp$) we then construct, by explicit linear--algebraic conditions, a closed subscheme of the naive splitting model (resp.\ of the splitting model) attached to the corresponding stratum. In \S \ref{LPBT}, we prove that these strata models are \'etale locally isomorphic to $\calZ^\Kra(\Lambda)$ (resp.\ $\calY^\Kra(\Lambda^\sharp)$) in $\overline{\calN}_n^\Kra$, and to their pull backs $\calZ(\Lambda)$ (resp.\ $\calY(\Lambda^\sharp)$) in $\overline{\calN}_n^\spl$. In particular, geometric ``local'' properties of the BT strata can be verified on these simpler models. 

\subsection{Review of the splitting model}\label{RSpl}

In this subsection we briefly recall the ramified unitary local model and the splitting model at a maximal vertex–stabilizer level; 
see \cite{Luo} for the local model and \cite{HLS1} for the splitting model.

Let $F/F_0$ be a ramified quadratic field extension with discrete valuation rings $O_F/O_{F_0}$. Let $\pi \in F$ (resp. $\pi_0$) be a uniformizer of $O_F$ (resp. $O_{F_0}$) with $\pi^2 = \pi_0$. Let $k$ be the perfect residue field of characteristic $\neq 2$.  Consider the $F$-vector space $F^n$ of dimension $n > 3$ and let 
\[
h: F^n \times F^n \rightarrow F
\]
be a split $F/F_0$-hermitian form, i.e. there is a basis $e_1, \dots, e_n$ of $F^n$ such that
\[
h(ae_i,be_{n+1-j}) = \overline{a}b\cdot \delta_{i,j} \quad \text{for  all} \quad a,b \in F, 
\]
where $a \mapsto \overline{a}$ is the non-trivial element of $\text{Gal}(F/{F_0})$. Attached to $h$ are the respective alternating and symmetric $F_0$-bilinear forms $F^n \times F^n \rightarrow F_0$ given by
\[
\langle x, y \rangle = \frac{1}{2}\text{Tr}_{F /F_0} (\pi^{-1}\phi(x, y)) \quad \text{and} \quad ( x, y ) =  \frac{1}{2}\text{Tr}_{F /F_0} (\phi(x, y)).
\]
For any $O_F$-lattice $\Lambda$ in $F^n$, we denote by $\Lambda^\sharp = \{v \in F^n | h( v, \Lambda ) \subset O_{F} \}$, $\Lambda^\vee = \{v \in F^n | \langle v, \Lambda \rangle \subset O_{F_0} \}$ and $\Lambda^\bot = \{v \in F^n | ( v, \Lambda ) \subset O_{F_0} \}$ the dual lattices respectively for the hermitian, alternating and symmetric forms.
The forms $\langle \, , \, \rangle $ and $ (\, , \,) $
induce perfect $O_{F_0}$-bilinear pairings
\begin{equation}\label{perfectpairing}
    \Lambda \times \Lambda^\vee \xrightarrow{\langle \, , \, \rangle } O_{F_0}, \quad \Lambda^\bot \times \Lambda \xrightarrow{ (\, , \,)} O_{F_0}
\end{equation}
for all $\Lambda$ and we have $\Lambda^\sharp= \Lambda^\vee=\pi \Lambda^\bot$.
For $i= kn+j$ with $0\leq j< n$, we define the standard lattices
\begin{equation}\label{eq 212}
\Lambda_i = \pi^{-k}\cdot \text{span}_{O_F} \{\pi^{-1}e_1, \dots, \pi^{-1}e_j, e_{j+1}, \dots, e_n\}.	
\end{equation} 
Note that $ \Lambda_{n-i}:=\Lambda^\bot_i $ and $ \Lambda_{-i}:=\Lambda^\sharp_i=\Lambda^\vee_i $. Then $\Lambda_i$'s form a self-dual periodic lattice chain $\mathcal{L}=\{\Lambda_i\}_{i\in \ZZ}$.
For nonempty subsets $I \subset \{0, \dots, m\}$ where $m = \lfloor n/2\rfloor$, 
let $\calL_I=\{\Lambda_i\}_{i\in \pm I+n\ZZ}$  be a self-dual periodic lattice chain. Let $\scrG_I = \underline{{\rm Aut}}(\mathcal{L}_I)$ be the (smooth) group scheme over $O_{F_0}$ with $P_I = \scrG_I(O_{F_0})$ the subgroup of $G(F_0)$ fixing the lattice chain $\mathcal{L}_I$ (see \cite[\S 1.2.3(b)]{PR} for more details). For even integers $2h, 2t$, where $0\leq 2h\neq 2t\leq n$, we define the following index sets: $[2h]=\pm h+n\ZZ$, $[2h, 2t]=\{\pm h, \pm t\}+n\ZZ$ and let $\calL_{[2h]}$, $\calL_{[2h,2t]}$ be the standard self-dual lattice chains.

In this paper, we always assume that $P_I$ is a (quasi)- parahoric stabilizer subgroup, i.e., $I=\{h\}$ for any $0\leq h\leq m$ and dimension $n$. We first recall the definition of the {\it wedge local model} $\M^{[2h],\wedge}_n$. For any $O_F$-algebra $R$, let $\Lambda_{i,R}$ be the tensor product $\Lambda_i\otimes_{O_{F_0}} R$ as an $O_F\otimes_{O_{F_0}} R$-module. Set $\Pi=\pi\otimes 1$ and $\pi=1\otimes \pi$.

\begin{Definition}
{\rm
The wedge local model $\M_n^{[2h], \wedge}$ is a projective scheme over $\Spec O_F$ representing the functor that sends each $O_F$-algebra $R$ to the set of subsheaves $\calF_i \subset  \Lambda_{i,R}$, where $i\in [2h]$, such that
\begin{itemize}
    \item For all $i\in [2h]$, $\calF_i$ as $O_F\otimes R$-modules are Zariski locally on $R$ direct summands of rank $n$.
    \item For all $i, j\in [2h]$ with $i<j$, the maps induced by the inclusions $  \Lambda_{i,R}\subset \Lambda_{j, R}$ restrict to maps
    \[
    \calF_i \rightarrow \calF_j.
    \]
    \item For all $i\in [2h]$, the isomorphism $\Pi: \Lambda_{i, R} \rightarroweq \Lambda_{i-n, R}$ identifies
    \[
    \calF_i \rightarroweq \calF_{i-n}.
    \]
    \item For all $i\in [2h]$, $\calF_{-i}$ is the orthogonal complement of $\calF_i$ with respect to $ \bb \,,\,\pp : \Lambda_{-i} \times  \Lambda_i\rightarrow R$.
    \item (Kottwitz condition) For all $i\in [2h]$,
    \[
    \text{char}_{\Pi |  \mathcal{F}_i} (X)= (X + \pi)^{n-1}(X - \pi) .
    \] 
    \item (Wedge condition) For all $i\in [2h]$,
    \[
    \wedge^2(\Pi-\pi\mid \calF_i)=0, \quad \wedge^n(\Pi+\pi\mid \calF_i)=0.
    \]
\end{itemize}   
}
\end{Definition}
The wedge local model $\M^{[2h],\wedge}_n$ is not always flat (see \cite{Sm3} for more details) and we define the {\it local model} $\M^{\loc, [2h]}_n$ as the flat closure of $\M^{[2h],\wedge}_n$. The following theorem is due to \cite{Arz}, \cite{Luo}, \cite{P},  \cite{Sm3}, \cite{Yu} among others (we refer to the introduction of \cite{Luo} for the precise references).

\begin{Theorem}\label{thmloc}
Let the signature $(n-s,s)=(n-1,1)$, the local model $\M^{\loc, [2h]}_n$ satisfies: 
\begin{itemize}
    \item[(1).] When $n=2m, h=m$ ($\pi$-modular case), the local model $\M^{\loc, [2h]}_n$ is smooth, connected, with dimension $n$;
    \item[(2).] When $h$ is not $\pi$-modular case, the local model $\M^{\loc, [2h]}_n$ is flat, normal and Cohen-Macaulay with dimension $n$. The special fiber of $\M^{\loc, [2h]}_n$ is reduced with dimension $n-1$. Each of its irreducible components is normal and Cohen-Macaulay. The number of irreducible components are:
    \begin{itemize}
        \item[(a)] When $h=0$, the special fiber of $\M^{\loc, [2h]}_n$ is irreducible;
        \item[(b)] When $n=2m+1$ and $h=m$, the special fiber of $\M^{\loc, [2h]}_n$ is smooth and irreducible;
        \item[(b)] When $n=2m$ and $h=m-1$, the special fiber of $\M^{\loc, [2h]}_n$ has three irreducible components;
        \item[(c)] For other cases, the special fiber of $\M^{\loc, [2h]}_n$ has two irreducible components.
    \end{itemize}
\end{itemize}
\end{Theorem}

Now we consider the splitting model $\M^{\spl, [2h]}_n$. The splitting model is a certain modification of the local model which ``refines" the ramification information by adding certain additional data to the local model. We first recall the naive splitting model $\M^{\nspl, [2h]}_n$.

\begin{Definition}\label{Def. 3.3}
{\rm
The naive splitting model $\M^{\nspl,[2h]}_n$ is a projective scheme over $\Spec O_F$ representing the functor that sends each $O_F$-algebra $R$ to the set of subsheaves 
\[
\begin{array}{l}
   \calG_i\subset \calF_i \subset  \Lambda_{i,R}, \quad ~\text{where}~ i\in [2h] 
\end{array}
\]
such that
\begin{itemize}
    \item For all $i\in [2h]$, $\calF_i$ (resp. $\calG_i$) as $O_F\otimes R$-modules are Zariski locally on $R$ direct summands of rank $n$ (resp. rank $1$).
    \item For all $i\in [2h]$, $(\calF_i)\in \M^{\loc, [2h]}_n\otimes R$.
    \item (Splitting condition)  For all $i\in [2h]$,
    \[
    (\Pi+\pi) \calF_{i} \subset \calG_{i}, \quad  (\Pi-\pi) \calG_{i} = (0).
    \]
\end{itemize}
}
\end{Definition}

The splitting model $\M^{\spl, [2h]}_n$ is the scheme-theoretic closure of the generic fiber $\M^{\nspl, [2h]}_{n,\eta}$ in $\M^{\nspl, [2h]}_n$. Both the local model $\M^{\loc,[2h]}_n$ and the naive splitting model $\M^{\nspl,[2h]}_n$ supports an action of $\scrG_{[2h]}$. By construction, there is a $\scrG_{[2h]}$-equivariant projective morphism 
\begin{equation}\label{eq tau}
\tau : \M^{\nspl,[2h]}_n \rightarrow \M^{\loc,[2h]}_n    
\end{equation}
which is given by $(\calF_i, \calG_{i}) \mapsto (\calF_i)$ on $R$-valued points. As in \cite[Definition 4.1]{Kr}, the forgetful morphism $\tau$ is well defined and induces an isomorphism on the generic fibers (see \cite[Lemma 2.4.1]{Luo}). Also, the morphism $\tau$ induces an isomorphism 
\[
\M^{\nspl,[2h]}_n\setminus \tau^{-1}(*) \cong  \M^{\loc,[2h]}_n\setminus\{*\} 
\]
(see \cite[Proposition 2.5.1]{HLS1}). 

\begin{Remarks}\label{rk 3.4}
{\rm
(1) The splitting model $\M^{\spl,[2h]}_n$ was first studied in \cite{Kr} for $h=0$. A moduli-theoretic description is given there for this case and is also known for $n=2m$ with $h=m$ \cite{ZacZhao2} and with $h=m-1$ \cite{Yu}; for general $h$ it remains unknown.

(2) When $h$ is not $\pi$-modular, there exists a special point $*=(\calF_i=\Pi\Lambda_{i,R})$ in the special fiber of the local model $\M^{\loc, [2h]}_n$, referred to as the {\it worst point}, which is the only closed Schubert cell that lies in the special fiber of the local model. Also, it is the only singular point of the irreducble components of the special fiber and of their intersections. We define the naive exceptional divisor $\NExc$ as $\NExc:=\tau^{-1}(*)\subset \M^{\nspl, [2h]}_n$, and the exceptional divisor $\Exc$ as the scheme-theoretic intersection $\Exc:=\tau^{-1}(*)\cap \M^{\spl, [2h]}_n$ (see \cite[Definition 2.5.2]{Luo}).

(3) When $h$ is $\pi$-modular, the local model does not possess a worst point. Moreover, in this case, the splitting model is isomorphic to the local model and the exceptional divisor $\Exc$ is empty (see \cite[Remark 2.5.3]{Luo}). Since the Bruhat-Tits stratification for this case has been treated by \cite{Wu}, we exclude this case and always assume that the index $I = \{h\}$ is not $\pi$-modular.
}\end{Remarks}

By Theorem \ref{thmloc}, the number of irreducible components of the special fiber of $\M^{\loc, [2h]}_n$ depends on the dimension $n$ and the level $I=\{h\}$. Let  $c=c(n,h)$ be the number of irreducible components and denote by $Z_i$ the irreducible components for $1\leq i\leq c$. 

\begin{Theorem}[\cite{Luo}, Theorem.~1.3.1]\label{thm splsemistable}
Assume $h$ is not $\pi$-modular. 
\begin{enumerate}
    \item The splitting model $\M^{\spl, [2h]}_n$ is flat and semi-stable of dimension $n$. 
    \item The special fiber of $\M^{\spl, [2h]}_n$ has $c+1$-irreducible components: $\tilde{Z}_i$ for $1\leq i\leq c$ and $\Exc$. These components are smooth and of dimension $n-1$;
    \item The forgetful morphism $\tau : \M^{\spl,[2h]}_n \rightarrow \M^{\loc,[2h]}_n$ maps $\tilde{Z}_i$ to $Z_i$ for $1\leq i\leq c$, and $\Exc$ to the worst point $*$.
\end{enumerate}
\end{Theorem}

Finally, we want to highlight that the splitting model $\M^{\spl, [2h]}_n$ is actually the blow-up of the local model $\M^{\loc, [2h]}_n$ along the worst point $*$.

\begin{Theorem}\label{thm 3.6}
Assume $h$ is not $\pi$-modular. The forgetful morphism $\tau : \M^{\spl,[2h]}_n \rightarrow \M^{\loc,[2h]}_n$ is the blow-up $ \M^{\bl,[2h]}_n$ of $\M^{\loc, [2h]}_n$ along the worst point $*$. Under this identification we have: a) $\Exc$ is the exceptional divisor of the blow-up and is isomorphic to the blow-up of $\PP_k^{n-1}$ along $\PP_k^{2h-1}$, and b) $\tilde{Z}_i$ is the strict transform
of $Z_i$ for $1\leq i\leq c$.
\end{Theorem}

\begin{proof}
See \cite[Theorem 4.0.1]{HLS1} and \cite[Proposition 4.3.2]{HLS1}.
\end{proof}

\subsection{Strata models}\label{sec. 32}
In this subsection, we will define the naive strata model $\M^{\nspl, [2h]}_n(2t)$ (resp. strata model $\calM^{[2h]}_n(2t)$). It is a closed subscheme in the special fiber of naive splitting model $\overline{\M}^{\nspl,[2h]}_n$ (resp. the special fiber of splitting model $\overline{\M}^{\spl,[2h]}_n$). 
 We first review the definition of the strata local model $\M^{\loc, [2h]}_n(2t)$ from \cite{HLS2}.

\begin{Definition}
{\rm
Let $R$ be a $k$-algebra, and $\calL_{[2h]}$, $\calL_{[2h,2t]}$ be the standard self-dual lattice chains. The strata local model $\M^{\loc, [2h]}_n(2t)$ is the projective scheme over $\Spec k$, representing the functor that sends each $k$-algebra $R$ to the set of subsheaves 
\[
\calF_i \subset  \Lambda_{i,R}, \quad \text{where}~ i\in [2h,2t] 
\]
such that
\begin{itemize}
    \item $\calF_i=\Pi \Lambda_i$ for $i\in [2t]$.
    \item $(\calF_i)_{i\in [2h]}\in \overline{\M}^{\loc, [2h]}_n\otimes R$.
    \item For any $i< j$ with either $i\in [2h], j\in [2t]$ or $i\in [2t], j\in [2h]$, the natural morphism $\Lambda_i\rightarrow \Lambda_j$ maps $\calF_i$ to $\calF_j$.
\end{itemize}
}
\end{Definition}

Here $\overline{\M}^{\loc, [2h]}_n$ is the special fiber of the local model $\M^{\loc, [2h]}_n$. The main theorem for the strata local model is the following:

\begin{Theorem}[\cite{HLS2}, \S 4]\label{StrataLocalModel}
The strata local model  $\M^{\loc, [2h]}_n(2t)$ is reduced, normal, and Cohen-Macaulay. Moreover, 

(1). For $t> h$, the strata local model $\M^{\loc, [2h]}_n(2t)$ has dimension $t+h$.

(2). For $t< h$, excluding the case where $n$ is even and $h=\frac n2$ ($\pi$-modular case), the strata local model $\M^{\loc, [2h]}_n(2t)$ has dimension $n-t-h-1$.

(3). For $n$ is even, $t< h=\frac n2$, the strata local model $\M^{\loc, [n]}_n(2t)$ is smooth and irreducible of dimension $\frac n2-t-1$.
\end{Theorem}

Now we can define the (naive) strata model. For even integers $2h, 2t$, where $0\leq 2h\neq 2t\leq n$, let $\overline{\M}^{\nspl,[2h]}_n$ be the special fiber of the naive splitting model over $\Spec k$.

\begin{Definition}\label{defstratanspl}
{\rm
Let $R$ be a $k$-algebra. The naive strata model $\M^{\nspl,[2h]}_n(2t)$ is the projective scheme over $\Spec k$, representing the functor that sends each $k$-algebra $R$ to the set of subsheaves 
\[
\begin{array}{l}
   \calF_i \subset  \Lambda_{i,R}, \quad ~\text{where}~ i\in [2h, 2t]  \\
   \calG_j \subset  \calF_i, \quad ~\text{where}~ j\in [2h]
\end{array}
\]
such that
\begin{itemize}
    \item $\calF_i=\Pi \Lambda_{i, R}$ for $i\in [2t]$.
    \item $(\calF_j, \calG_j)_{j\in [2h]}\in \overline{\M}^{\nspl, [2h]}_n\otimes R$.
    \item For any $i_1< i_2$ with either $i_1\in [2h], i_2\in [2t]$ or $i_1\in [2t], i_2\in [2h]$, the natural morphism $\Lambda_{i_1}\rightarrow \Lambda_{i_2}$ maps $\calF_{i_1}$ to $\calF_{i_2}$.
    \item When $t> h$, we have 
    \[
    \calG_j\subset L_\pm^\bot ~\text{for}~ j=\pm h+kn.
    \]
    Here $L_\pm$ is the image of $\Lambda_{-t+kn}\rightarrow \Lambda_{\mp h+kn}$, and $L_\pm^\bot$ is the dual of $L_\pm$ with respect to $\bb~,~ \pp$.
    \item When $t< h$, we have 
    \[
    \begin{array}{c}
    \calG_j\subset L_+^\bot ~\text{for}~ j= h+kn.\\
    (\text{resp.}~ \calG_j\subset L_-^\bot ~\text{for}~ j=(n-h)+kn. )
    \end{array}
    \]
     Here $L_+$ is the image of $\Lambda_{t+kn}\rightarrow \Lambda_{(n-h)+kn}$ (resp. $L_-$ is the image of $\Lambda_{t+kn}\rightarrow \Lambda_{h+kn}$), and $L_\pm^\bot$ is the dual of $L_\pm$ with respect to $( ~ ,~ )$.
\end{itemize}
}
\end{Definition}

The naive strata model $\M^{\nspl,[2h]}_n(2t)$ is a  closed subscheme of $\overline{\M}^{\nspl,[2h]}_n$. By restricting the morphism $\tau$ in (\ref{eq tau}) to the strata model, we get a projective $\scrG_{[2h]}$-equivariant morphism $ \tau: \M^{\nspl,[2h]}_n(2t)\rightarrow \M^{\loc,[2h]}_n(2t)$. We now define the strata model $\calM^{[2h]}_n(2t)$—a closed subscheme of the special fiber of $\M^{\spl,[2h]}_n$—as follows:

\begin{Definition}
Let $R$ be a $k$-algebra. The strata model $\calM^{[2h]}_n(2t)$ is the projective scheme over $\Spec k$, whose $R$-points are 
\[
\calM^{[2h]}_n(2t)(R)=\M^{\spl, [2h]}_n (R)\cap \M^{\nspl,[2h]}_n(2t)(R).
\]
\end{Definition}

The main theorem of this section is as follows.

\begin{Theorem}\label{thm splstrata}
The strata model  $\calM^{[2h]}_n(2t)$ is smooth. Moreover, 

(1). For $t> h$, the strata model $\calM^{[2h]}_n(2t)$ is of dimension $t+h$.

(2). For $t< h$, excluding the case where $n$ is even and $h=\frac n2$ ($\pi$-modular case), the strata model $\calM^{[2h]}_n(2t)$ is of dimension $n-t-h-1$.

(3). For $t_2<h<t_1$, the intersection of strata models $\calM^{[2h]}_n(2t_1)\cap \calM^{[2h]}_n(2t_2)$ is of dimension $t_1-t_2-1$.   
\end{Theorem}

\subsection{Affine charts for the splitting model}\label{AffineCharts3.3}
In this subsection, following the discussion in \cite[\S 3]{Luo} we review the main steps of the proof of Theorem \ref{thm splsemistable}. By \cite[\S 3.1]{Luo}, it suffices to compute an open affine covering of the inverse image of the worst point $*$ under $\tau$, i.e., the $\NExc$ for the naive splitting model, and $\Exc$ for the splitting model. Assume that the index $I=\{h\}$ is not $\pi$-modular.

To simplify the computations below, we fix reordered bases of $\Lambda_{h}$ and $\Lambda_{n-h}$ as in \cite[(4.1.1)]{ZacZhao2} (see also \cite[(3.1.2)]{HLS1}):  
\begin{equation}\label{eq 411}
\Lambda_{h}=\operatorname{span}_{O_{F_0}}\left\{
\begin{array}{l}
e_{n-h+1},\cdots,e_n,\ \pi^{-1}e_1,\cdots,\pi^{-1}e_h,\ e_{h+1},\cdots,e_{n-h},\\
\pi e_{n-h+1},\cdots,\pi e_n,\ e_1,\cdots,e_h,\ \pi e_{h+1},\cdots,\pi e_{n-h}
\end{array}
\right\},
\end{equation}
\begin{equation*}
\Lambda_{n-h}=\operatorname{span}_{O_{F_0}}\left\{
\begin{array}{l}
e_{n-h+1},\cdots,e_n,\ \pi^{-1}e_1,\cdots,\pi^{-1}e_h,\ \pi^{-1}e_{h+1},\cdots,\pi^{-1}e_{n-h},\\
\pi e_{n-h+1},\cdots,\pi e_n,\ e_1,\cdots,e_h,\ e_{h+1},\cdots,e_{n-h}
\end{array}
\right\}.
\end{equation*}

We choose an open affine chart $\U_{i_0, j_0}$ of a point $(\Pi\Lambda_h, \Pi \Lambda_{n-h}, \calG_h, \calG_{n-h})$ in $\NExc$. In particular, this chart contains the points in 
\begin{equation}
\begin{tikzcd}
\Lambda_{h,R} \arrow[r,"\lambda_h"] 
&
\Lambda_{n-h,R} \arrow[r,"\lambda_{n-h}"] 
&
\Pi^{-1} \Lambda_{h,R} 
\\
\calF_h \arrow[u,hook]\arrow{r}
&
\calF_{n-h} \arrow[u,hook]\arrow{r}
&
\Pi^{-1}\calF_{h} \arrow[u,hook]
\\
\calG_{h}\arrow[u, hook]\arrow{r}&
\calG_{-h}\arrow[u, hook]\arrow{r}&
\Pi^{-1} \calG_{h}\arrow[u, hook]
\end{tikzcd},
\end{equation}
where
\begin{equation}\label{eq FG}
\calF_h = \left[\ 
\begin{matrix}
X\\ \hline
I_n  
\end{matrix}\ \right], \quad 
\mathcal{F}_{n-h}=  \left[\ 
\begin{matrix}
Y\\ \hline
I_n  
\end{matrix}\ \right],\quad
\calG_{h}= \left[\ 
\begin{matrix}
\pi T\\ \hline
T
\end{matrix}\ \right], \quad
\calG_{n-h}= \left[\ 
\begin{matrix}
\pi V\\ \hline
V  
\end{matrix}\ \right] .  
\end{equation}
Here $X, Y$ are matrices of size $n\times n$ and the matrices $T, V$ are of size $n\times 1$ with respect to the reordered basis. We break up the matrices $X, Y, T, V$ into blocks as follows:
\begin{equation}
X= \begin{blockarray}{ccc}
\matindex{2h} & \matindex{n-2h}\\
\begin{block}{[cc]c}
X_1  & X_2 & \matindex{2h} \\ 
X_3  & X_4 &\matindex{n-2h} \\ 
\end{block}
   \end{blockarray},\
Y= \begin{blockarray}{ccc}
\matindex{2h} & \matindex{n-2h}\\
\begin{block}{[cc]c}
Y_1  & Y_2 & \matindex{2h} \\ 
Y_3  & Y_4 &\matindex{n-2h} \\ 
\end{block}
   \end{blockarray},\
T= \begin{blockarray}{cc}
\matindex{1} & \\
\begin{block}{[c]c}
T_1  & \matindex{2h} \\ 
T_2  & \matindex{n-2h} \\ 
\end{block}
   \end{blockarray},\
V= \begin{blockarray}{cc}
\matindex{1} & \\
\begin{block}{[c]c}
V_1  & \matindex{2h} \\ 
V_2  & \matindex{n-2h} \\ 
\end{block}
   \end{blockarray}.
\end{equation}
By \cite[\S 3.2.3]{HLS1}, there exist $n\times 1$ matrices $S= \left[\ 
\begin{matrix}
S_1\\ 
S_2 
\end{matrix}\ \right]$ and $Z= \left[\ 
\begin{matrix}
Z_1\\ 
Z_2 
\end{matrix}\ \right]$, with $S_i$ (resp. $Z_i$) of the same size as $T_i$ (resp. $V_i$) for $i=1,2$, such that $X$ (resp. $Y$) can be expressed in terms of $T, S$ (resp. $V, Z$):
\begin{equation}\label{eq 334}
\begin{array}{l}
X=TS^t-\pi I_n,\quad Y=VZ^t-\pi I_n.
\end{array}
\end{equation}	
For the naive splitting model $\M^{\nspl, [2h]}_n$, we have an open covering 
\[
\NExc\subset \cup_{1\leq i_0, j_0\leq n} ~\U_{i_0, j_0},
\]
where a point $(\calF_h, \calF_{n-h}, \calG_h, \calG_{n-h})$ in $\U_{i_0, j_0}$ satisfy $t_{i_0}=1$ for $\calG_h$ (resp. $ v_{j_0}=1$ for $\calG_{n-h}$). Set 
\[
i^\vee:=
\begin{cases}
  2h+1-i,   & \text{if } 1\le i\le 2h,\\
  n+2h+1-i, & \text{if } 2h+1\le i\le n .
\end{cases}
\]
Let $H=H_{n-2h}$ be the unit anti-diagonal matrix of size $n-2h$. By \cite[\S 3]{HLS1}, there are four cases of $\U_{i_0, j_0}$.

\begin{Proposition}\label{affinechartU}
(Case 1) For $1\leq j_0\leq 2h < i_0\leq n$, the affine chart $\U_{i_0, j_0}$ is isomorphic to
\[
\Spec \frac{O_F[V_1, T_2, v_{i_0}, t_{j_0}, z_{i_0^\vee}]}{(t_{i_0}-1, v_{j_0}-1, v_{i_0}t_{j_0}-\pi, z_{i_0^\vee} v_{i_0}(T_2^tHT_2)-2\pi, z_{i_0^\vee}(z_{i_0^\vee} (T_2^tHT_2)-2t_{j_0}))}.
\]
(Case 2) For $1\leq i_0, j_0\leq 2h$, the affine chart $\U_{i_0, j_0}$ is isomorphic to
\[
\Spec \frac{O_F[T_1, T_2, s_{i_0^\vee}, t_{j_0}^{-1}]}{(t_{i_0}-1,  s_{i_0^\vee}(s_{i_0^\vee}(T_2^tHT_2)-2), \pi(s_{i_0^\vee}(T_2^tHT_2)-2) )}.
\]
(Case 3) For $2h+1\leq i_0, j_0\leq n$, the affine chart $\U_{i_0, j_0}$ is isomorphic to
\[
\Spec \frac{O_F[V_1, V_2, z_{j_0^\vee} , v_{i_0}^{-1}]}{(v_{j_0}-1,  z_{j_0^\vee}(V_2^tHV_2)-2\pi )}.
\]
(Case 4) For $1\leq i_0\leq 2h < j_0\leq n$, the affine chart $\U_{i_0, j_0}$ is smooth and contained in the generic fiber.
\end{Proposition}

\begin{proof}
See \cite[Propositions 3.5.3, 3.6.2, 3.7.3]{HLS1} for the proof of Case (1), (2), and (3). Case (4) follows from \cite[\S 3.8]{HLS1} 
\end{proof}

It is easy to see that $\U_{i_0, j_0}$ is not flat in the first two cases. Let $\U_{i_0, j_0}'=\U_{i_0, j_0}\cap \M^{\spl, [2h]}_n$ denote its flat closure, we obtain

\begin{Proposition}\label{prop splflat}
(Case 1) For $1\leq j_0\leq 2h < i_0\leq n$, the affine chart $\U_{i_0, j_0}'$ is isomorphic to
\[
\Spec \frac{O_F[V_1, T_2, v_{i_0}, z_{i_0^\vee}]}{(t_{i_0}-1, v_{j_0}-1, z_{i_0^\vee} v_{i_0}(T_2^tHT_2)-2\pi)}.
\]
$\U_{i_0, j_0}'$ is semi-stable over $O_F$. In the special fiber, it has 2 irreducible components when $n=2m+1, h=m$ (almost $\pi$-modular case), and 3 irreducible components when $h$ is not almost $\pi$-modular case.

(Case 2) For $1\leq i_0, j_0\leq 2h$, the affine chart $\U_{i_0, j_0}'$ is isomorphic to
\[
\Spec \frac{O_F[T_1, T_2, s_{i_0^\vee}, t_{j_0}^{-1}]}{(t_{i_0}-1,  s_{i_0^\vee}(T_2^tHT_2)-2 )},
\]
such that $\U_{i_0, j_0}'$ is smooth.

(Case 3) For $2h+1\leq i_0, j_0\leq n$, the affine chart $\U_{i_0, j_0}'=\U_{i_0, j_0}$ is flat and semi-stable over $O_F$, with 2 irreducible components in the special fiber.

(Case 4) For $1\leq i_0\leq 2h < j_0\leq n$, the affine chart $\U_{i_0, j_0}'=\U_{i_0, j_0}$ is smooth.
\end{Proposition}

\begin{proof}
See \cite[Theorems 3.5.5, 3.6.3, 3.7.4, Remark 3.5.6, 3.6.4]{HLS1} for the proof of Case (1), (2), and (3). Case (4) follows from the fact that $\calU_{i_0, j_0}$ is contained in the generic fiber.   
\end{proof}

\subsection{Affine charts for strata models}\label{affinechart strata} In this subsection, we will calculate the strata models $\calM^{[2h]}_n(2t)$ for $t>h$ ($\calZ$-strata case) and $t<h$ ($\calY$-strata case). We always assume that $\Lambda\subset C$ is a vertex lattice of type $2t$ with $t\neq h$.

\subsubsection{$\calZ$-strata models}
We begin by computing the strata models $\calM^{[2h]}_n(2t)$ for $t>h$. 
By an unramified extension, we can reduce to the case where the hermitian form is split. 
By Definition \ref{defstratanspl}, when $t>h$, the strata model parameterizes quadruples $(\calF_h, \calF_{-h}, \calG_{h}, \calG_{-h})$ fitting into the diagram:
\begin{equation}
\begin{tikzcd}
\Lambda_{-t,R} \arrow[r,"\lambda_1"] 
&
\Lambda_{-h,R} \arrow[r,"\lambda"] 
&
\Lambda_{h,R} \arrow[r,"\lambda_2"]
&
\Lambda_{t,R}
\\
\Pi\Lambda_{-t,R} \arrow[u,hook]\arrow{r}
&
\calF_{-h} \arrow[u,hook]\arrow{r}
&
\calF_{h} \arrow[u,hook]\arrow{r}
&
\Pi\Lambda_{t,R}\arrow[u,hook]
\\
&
\calG_{-h}\arrow[u, hook]\arrow{r}
&\calG_{h}\arrow[u, hook]
&
\end{tikzcd}.
\end{equation}
Let $(\calF_h, \calF_{-h}, \calG_{h}, \calG_{-h})$ be an $R$-point in $\overline{\M}^{\nspl, [2h]}_n$ with the represented matrices given by (\ref{eq FG}). (For the undefined terms below we refer to \S \ref{AffineCharts3.3}.) Note that $\calF_{-h}\simeq \calF_{n-h}$, $\calG_{-h}\simeq \calG_{n-h}$ with the same represented matrices. To obtain $(\calF_h, \calF_{-h}, \calG_{h}, \calG_{-h})\in \M^{\nspl, [2h]}_n(2t)(R)$, we need to check:
\begin{equation}
\lambda_1(\Pi \Lambda_{-t,R})\subset \calF_{-h},\quad
\lambda_2(\calF_h)\subset \Pi \Lambda_{t,R},\quad
\calG_{h}\subset L_+^\bot, \quad
\calG_{-h}\subset L_-^\bot, 
\end{equation}
where $L_\pm$ is the image of $\Lambda_{-t}\rightarrow \Lambda_{\mp h}$. Note that the reordered basis of $\Lambda_{\pm t}$, $\Lambda_{\pm h}$ is the same as \cite[(4.1.1)]{ZacZhao2}. With respect to this reordered basis, we have:
\small{\begin{equation}\label{eq trans}
\lambda_1=\left[\ \begin{array}{ll|ll}
0 & I_{n-t+h} &0 & 0 \\ 
0_{t-h}& 0 &0 &0\\ \hline
0&0_{n-t+h}&0&I_{n-t+h}\\
I_{t-h}&0&0_{t-h}&0
\end{array}\ \right],\quad	
\lambda_2=\left[\ \begin{array}{ll|ll}
0 & I_{t-h} &0 & 0 \\ 
A& 0 &0 &0\\ \hline
0&0_{t-h}&0&I_{t-h}\\
B&0&A&0
\end{array}\ \right],
\end{equation}}
where
\[
A=\left[\ \begin{matrix}
I_{2h}  &0 & 0 \\ 
 0 &0_{t-h} &0\\
0&0&I_{n-2t}
\end{matrix}\ \right],\quad
B=\left[\ \begin{matrix}
0_{2h}  &0 & 0 \\ 
 0 &I_{t-h} &0\\
0&0&0_{n-2t}
\end{matrix}\ \right].
\]
Now $\lambda_1(\Pi \Lambda_{-t,R})\subset \calF_{-h}$ implies that
\begin{equation}\label{eq Y}
Y_1=0,\quad Y_3=0,\quad Y_2=\begin{blockarray}{cccc}
\matindex{t-h} &\matindex{n-2t}&\matindex{t-h}&\\
\begin{block}{[ccc]c}
  0  &0 & * & \matindex{h} \\ 
 0 &0&*& \matindex{h} \\
\end{block}
   \end{blockarray}
,\quad
Y_4=\begin{blockarray}{cccc}
\matindex{t-h} &\matindex{n-2t}&\matindex{t-h}&\\
\begin{block}{[ccc]c}
  0  &0 & * & \matindex{t-h} \\ 
 0 &0&*& \matindex{n-2t} \\
 0  &0 & * & \matindex{t-h} \\ 
\end{block}
   \end{blockarray}.
\end{equation}
The coordinates of $V_1, V_2, Z_1, Z_2$ can be further refined as follows:
\begin{equation}\label{eq VZ}
V_1= \begin{blockarray}{cc}
\matindex{1} & \\
\begin{block}{[c]c}
V_{1,1}  & \matindex{h} \\ 
V_{1,2}  & \matindex{h} \\ 
\end{block}
   \end{blockarray},
V_2= \begin{blockarray}{cc}
\matindex{1} & \\
\begin{block}{[c]c}
V_{2,1}  & \matindex{t-h} \\ 
V_{2,2}  & \matindex{n-2t} \\ 
V_{2,3}  & \matindex{t-h} \\ 
\end{block}
   \end{blockarray},
Z_1= \begin{blockarray}{cc}
\matindex{1} & \\
\begin{block}{[c]c}
Z_{1,1}  & \matindex{h} \\ 
Z_{1,2}  & \matindex{h} \\ 
\end{block}
   \end{blockarray},
Z_2= \begin{blockarray}{cc}
\matindex{1} & \\
\begin{block}{[c]c}
Z_{2,1}  & \matindex{t-h} \\ 
Z_{2,2}  & \matindex{n-2t} \\ 
Z_{2,3}  & \matindex{t-h} \\ 
\end{block}
   \end{blockarray}.
\end{equation}
Thus, we obtain 
\begin{equation}\label{eq 349}
\left[\ 
\begin{matrix}
V_1\\ 
V_2
\end{matrix}\ \right]\cdot 
\left[\ 
\begin{matrix}
Z_{1,1}^t, Z_{1,2}^t, Z_{2,1}^t, Z_{2,2}^t\\ 
\end{matrix}\ \right]=0    
\end{equation}
by combining equations in (\ref{eq 334}), (\ref{eq Y}). Since there exists a unit element in $V$, equation (\ref{eq 349}) is equivalent to 
\begin{equation}\label{eq 3410}
Z_1=0,\quad  Z_{2,1}=0,\quad Z_{2,2}=0.
\end{equation}
Similarly, condition $\lambda_2(\calF_h)\subset \Pi \Lambda_{t,R}$ gives the same equation as in (\ref{eq 3410}) by $\bb \calF_h, \calF_{-h}\pp=0$. Now consider $\calG_{\pm h}\subset L_{\pm}^\bot$, where where $L_\pm$ is the image of $\Lambda_{-t}\rightarrow \Lambda_{\mp h}$, i.e.,
\quash{
\begin{equation}\label{eq X}
X_1=0,\quad X_2=0,\quad X_3=\left[\ \begin{matrix}
* & * \\ 
 0 &0\\
 0& 0
\end{matrix}\ \right],\quad
X_4=\left[\ \begin{matrix}
*  &*& * \\ 
 0 &0 &0\\
0&0&0
\end{matrix}\ \right]. 
\end{equation}}
\begin{equation}
\begin{array}{c}
L_+=\text{span}_{R}\{e_1,\cdots,e_{n-t}, \pi e_{1},\cdots, \pi e_n, \pi_0 e_{n-h+1}, \cdots, \pi_0 e_n\}\subset \Lambda_{-h},  \\
L_-=\text{span}_{R}\{e_1,\cdots,e_{n-t}, \pi e_{h+1},\cdots, \pi e_n\}\subset \Lambda_h, 
\end{array}
\end{equation}
of rank $2n-(t-h)$ and $2n-(h+t)$. The dual lattices of $L_\pm$ with respect to $\bb ~ , ~\pp$ are
\begin{equation}
\begin{array}{c}
L_+^\bot=\text{span}_{R}\{\pi e_{h+1},\cdots,\pi e_{t}, \} \subset \Lambda_{h}    \\
L_-^\bot=\text{span}_{R}\{\pi e_1,\cdots,\pi e_{t}, p e_{n-h+1},\cdots, p e_n\}\subset \Lambda_{-h},
\end{array}
\end{equation}
of rank $t-h$ and $t+h$. We refine the coordinates of $T_1, T_2, S_1, S_2$ similar to (\ref{eq VZ}):
\begin{equation}
T_1= \begin{blockarray}{cc}
\matindex{1} & \\
\begin{block}{[c]c}
T_{1,1}  & \matindex{h} \\ 
T_{1,2}  & \matindex{h} \\ 
\end{block}
   \end{blockarray},
T_2= \begin{blockarray}{cc}
\matindex{1} & \\
\begin{block}{[c]c}
T_{2,1}  & \matindex{t-h} \\ 
T_{2,2}  & \matindex{n-2t} \\ 
T_{2,3}  & \matindex{t-h} \\ 
\end{block}
   \end{blockarray},
S_1= \begin{blockarray}{cc}
\matindex{1} & \\
\begin{block}{[c]c}
S_{1,1}  & \matindex{h} \\ 
S_{1,2}  & \matindex{h} \\ 
\end{block}
   \end{blockarray},
S_2= \begin{blockarray}{cc}
\matindex{1} & \\
\begin{block}{[c]c}
S_{2,1}  & \matindex{t-h} \\ 
S_{2,2}  & \matindex{n-2t} \\ 
S_{2,3}  & \matindex{t-h} \\ 
\end{block}
   \end{blockarray}.
\end{equation}
Thus, the condition $\calG_{\pm h}\subset L_{\pm}^\bot$ translate to 
\begin{equation}
\left[\ 
\begin{matrix}
0_n\\ \hline
T_{1,1}\\
T_{1,2}\\
T_{2,1}\\
T_{2,2}\\
T_{2,3}\\
\end{matrix}\ \right]\subset
L_+^\bot=\left[\begin{matrix}
0_n\\ \hline
0 \\
0 \\
I_{t-h}\\
0 \\
0 \\
\end{matrix}\ \right],\quad
\left[\ 
\begin{matrix}
0_n\\ \hline
V_{1,1}\\
V_{1,2}\\
V_{2,1}\\
V_{2,2}\\
V_{2,3}\\
\end{matrix}\ \right]\subset
L_-^\bot=\left[\begin{matrix}
&0_n&\\ \hline
I_h &&\\
&I_h&\\
&& I_{t-h}\\
0&0&0\\
0&0&0\\
\end{matrix}\ \right].
\end{equation}
Therefore, we get 
\[
\begin{array}{cccc}
    T_{1,1}=0,& T_{1,2}=0,& T_{2,2}=0,& T_{2,3}=0   \\
   &  V_{2,2}=0,& V_{2,3}=0.& 
\end{array}
\]
Note that the invertible element in $T$ can only occur in $T_{2,1}$, i.e., $2h+1\leq i_0\leq t+h$. We only need to consider the affine charts $\calU_{i_0, j_0}$ in case (1), (3) of Proposition \ref{affinechartU}.

\begin{Proposition}\label{prop X-strata}
 When $1\leq j_0\leq 2h < i_0\leq t+h$, or $2h+1\leq i_0, j_0\leq t+h$, the open affine chart $U_{i_0, j_0}=\calU_{i_0, j_0}\cap \M^{\nspl, [2h]}_n(2t)$ in the naive strata model  $\M^{\nspl, [2h]}_n(2t)$ $(t> h)$ is isomorphic to the open affine chart $U_{i_0, j_0}'=\calU_{i_0, j_0}\cap \calM^{ [2h]}_n(2t)$ in the strata model  $\calM^{[2h]}_n(2t)$. Both are smooth and isomorphic to $\mathbb{A}^{t+h}_k$.
\end{Proposition}

\begin{proof}
For $1\leq j_0\leq 2h < i_0\leq t+h$, since $T_{2,2}=0$ and $T_{2,3}=0$, we have $T_2^tHT_2=0$. Note that $t_{j_0}\in T_{1,1}$ or $T_{1,2}$ by $1\leq j_0\leq 2h$. Thus we obtain $t_{j_0}=0$. By Proposition \ref{affinechartU} (1), the open affine chart $U_{i_0, j_0}=\calU_{i_0, j_0}\cap \M^{\nspl, [2h]}_n(2t)$ is isomorphic to
\[
\Spec \frac{k[V_{1,1}, V_{1,2}, T_{2,1}, v_{i_0}, z_{i_0^\vee}]}{(t_{i_0}-1, v_{j_0}-1)}\simeq \mathbb{A}^{t+h}_k.
\]
Similarly, it is easy to see that the open affine chart $U_{i_0, j_0}'=\calU_{i_0, j_0}\cap \calM^{ [2h]}_n(2t)$ is also isomorphic to $\mathbb{A}^{t+h}_k$ by Proposition \ref{prop splflat} (1).

For $2h+1\leq i_0, j_0\leq t+h$, we get $V_2^tHV_2=0$ by $V_{2,2}=0, V_{2,3}=0$. This gives $U_{i_0, j_0}=U'_{i_0, j_0}\simeq \mathbb{A}^{t+h}_k$ by Propositions \ref{affinechartU}, \ref{prop splflat} (3).
\end{proof}

\subsubsection{$\calY$-strata models}
For the $\calY$-strata models, let $\Lambda\subset C$ be a vertex lattice of type $2t$ with $t< h$. (Recall from Remark \ref{rk 3.4} that we exclude the $\pi$-modular case, i.e. even $n$ and $h=\frac n2$.)  
When $t<h$, the strata model parameterizes quadruples $(\calF_h, \calF_{n-h}, \calG_{h}, \calG_{n-h})$ fitting into the diagram:
\begin{equation}
\begin{tikzcd}
\Lambda_{t,R} \arrow[r,"\lambda_2"] 
&
\Lambda_{h,R} \arrow[r,"\lambda^\vee"] 
&
\Lambda_{n-h,R} \arrow[r,"\lambda_1"]
&
\Lambda_{n-t,R}
\\
\Pi\Lambda_{t,R} \arrow[u,hook]\arrow{r}
&
\calF_{h} \arrow[u,hook]\arrow{r}
&
\calF_{n-h} \arrow[u,hook]\arrow{r}
&
\Pi\Lambda_{n-t,R}\arrow[u,hook]
\\
&\calG_{h}\arrow[u, hook]\arrow{r}&\calG_{n-h}\arrow[u, hook]&
\end{tikzcd}.
\end{equation}
We express $\calF_h,\ \calF_{-h},\ \calG_h,\ \calG_{-h}$ using the same block decomposition as in (\ref{eq FG}). To obtain $(\calF_h, \calF_{n-h}, \calG_{h}, \calG_{n-h})\in \M^{\nspl, [2h]}_n(2t)(R)$, we need to check:
\begin{equation}
\lambda_2(\Pi \Lambda_{t,R})\subset \calF_{h},\quad
\lambda_1(\calF_{n-h})\subset \Pi \Lambda_{n-t,R},\quad
\calG_h\subset L_+^\bot,\quad
\calG_{n-h}\subset L_-^\bot.
\end{equation}
Here $L_+$ ( resp. $L_-$) is the image of $\Lambda_{t}\rightarrow \Lambda_{n-h}$ (resp. $\Lambda_{t}\rightarrow \Lambda_h$), and $L_\pm^\bot$ is the dual lattice with respect to $( ~ , ~)$. Similar to (\ref{eq trans}), the transition maps can be expressed as:
\begin{equation}
\lambda_1=\left[\ \begin{array}{ll|ll}
0 & I_{n-h+t} &0 & 0 \\ 
0_{h-t}& 0 &0 &0\\ \hline
0&0_{n-h+t}&0&I_{n-h+t}\\
I_{h-t}&0&0_{h-t}&0
\end{array}\ \right],\quad	
\lambda_2=\left[\ \begin{array}{ll|ll}
0 & I_{h-t} &0 & 0 \\ 
A& 0 &0 &0\\ \hline
0&0_{h-t}&0&I_{h-t}\\
B&0&A&0
\end{array}\ \right]
\end{equation}
where
\[
A=\left[\ \begin{matrix}
I_{2t}  &0 & 0 \\ 
 0 &0_{h-t} &0\\
0&0&I_{n-2h}
\end{matrix}\ \right],\quad
B=\left[\ \begin{matrix}
0_{2t}  &0 & 0 \\ 
 0 &I_{h-t} &0\\
0&0&0_{n-2h}
\end{matrix}\ \right].
\]
Condition $\lambda_1(\calF_{n-h})\subset \Pi \Lambda_{n-t,R}$ is equivalent to
\begin{equation}\label{eq Y2}
Y_1=\begin{blockarray}{ccccc}
\matindex{h-t} &\matindex{t}&\matindex{t}& \matindex{h-t}&\\
\begin{block}{[cccc]c}
  *  &* & * & *& \matindex{h-t} \\ 
 0 &0&0&0& \matindex{t} \\
 0  &0 & 0 & 0& \matindex{t} \\ 
 0 &0&0&0& \matindex{h-t} \\
\end{block}
   \end{blockarray},\quad
Y_2=\begin{blockarray}{cc}
\matindex{n-2h} &\\
\begin{block}{[c]c}
  * & \matindex{h-t} \\ 
 0 & \matindex{t} \\
 0 & \matindex{t} \\ 
 0 & \matindex{h-t} \\ 
\end{block}
   \end{blockarray},\quad
Y_3=0,\quad Y_4=0.
\end{equation}
Note that $X_1=\left[\ 
\begin{matrix}
A & B \\ 
C& D
\end{matrix}\ \right]$ as in \cite[(3.1.3)]{HLS1}. In the special fiber, we have relations
\[
D+HA^tH=0,\quad B=HB^tH,\quad C=HC^tH
\]
by \cite[Lemma 3.3.1]{HLS1}. Thus, we have $X_1=-JX_1^tJ$, where 
$J=\left[\ 
\begin{matrix}
&H \\ 
-H& 
\end{matrix}\ \right]$. By $(\calF_{n-h}, \calF_h)=0$, we also obtain $Y_1=-JX_1^tJ$ and so $Y_1=JY_1^tJ$. This implies $Y_1=(y_{i,j})_{1\leq i,j\leq 2h}$ satisfying $y_{i,j}=\pm y_{2h+1-j, 2h+1-i}$ for $1\leq i, j\leq 2h$. Therefore, the matrix $Y_1$ can be rewritten as 
\[
Y_1=\begin{blockarray}{ccccc}
\matindex{h-t} &\matindex{t}&\matindex{t}& \matindex{h-t}&\\
\begin{block}{[cccc]c}
  0  &0 & 0 & *& \matindex{h-t} \\ 
 0 &0&0&0& \matindex{t} \\
 0  &0 & 0 & 0& \matindex{t} \\ 
 0 &0&0&0& \matindex{h-t} \\
\end{block}
   \end{blockarray}.
\]
The coordinates of $V_1, T_1, S_1$ and $ Z_1$ can be further refined as follows:
\begin{equation}
V_1= \left[\begin{matrix}
V'_{1,1}   \\ 
V'_{1,2}   \\ 
V'_{1,3}  \\ 
V'_{1,4}    
\end{matrix}\right],\quad
T_1= \left[\begin{matrix}
T'_{1,1}   \\ 
T'_{1,2}   \\ 
T'_{1,3}  \\ 
T'_{1,4}    
\end{matrix}\right],\quad
S_1= \left[\begin{matrix}
S'_{1,1}   \\ 
S'_{1,2}   \\ 
S'_{1,3}  \\ 
S'_{1,4}    
\end{matrix}\right],\quad
Z_1= \left[\begin{matrix}
Z'_{1,1}   \\ 
Z'_{1,2}   \\ 
Z'_{1,3}  \\ 
Z'_{1,4}    
\end{matrix}\right],
\end{equation}
where $V_{1,i}'$ (resp. $T_{1,i}'$, $S_{1,i}'$, $Z_{1,i}'$) is of size $(h-t)\times 1$ for $i=1,4$, and $V_{1,j}'$ (resp. $T_{1,j}'$, $S_{1,j}'$, $Z_{1,j}'$) is of size $t\times 1$ for $j=2,3$.
From the first three columns of $Y_1, Y_3$, we get
\begin{equation}\label{eq 3322}
\left[\ 
\begin{matrix}
V_1\\ 
V_2
\end{matrix}\ \right]\cdot 
\left[\ 
\begin{matrix}
(Z_{1,1}')^{t} ~ (Z_{1,2}')^{t} ~(Z_{1,3}')^{t}\\ 
\end{matrix}\ \right]=0.  
\end{equation}
Since there exists a unit element in $V$, equation (\ref{eq 3322}) is equivalent to \
\begin{equation}\label{eq Z'}
Z_{1,1}'=0,\quad Z_{1,2}'=Z_{1,3}'=0.   
\end{equation} 
Similarly, condition $\lambda_2(\Pi \Lambda_{t,R})\subset \calF_{h}$ gives us the same relations as in (\ref{eq Z'}), so it suffices to verify $\calG_{h}\subset L_+^\bot, ~\calG_{n-h}\subset L_-^\bot$. Here $L_{+}$ (resp. $L_-$) is the image of $\Lambda_{t}\rightarrow \Lambda_{n-h}$ (resp. $\Lambda_{t}\rightarrow \Lambda_{h}$). More precisely,
\begin{equation}
\begin{array}{c}
     L_+=\text{span}_{R}\{\pi^{-1} e_1,\cdots,\pi^{-1} e_{t}, e_{1},\cdots, e_n, \pi e_{n-h+1}, \cdots, \pi e_n\}\subset \Lambda_{n-h}  \\
    (\text{resp.}~ L_-=\text{span}_{R}\{\pi^{-1} e_1,\cdots,\pi^{-1} e_{t}, e_{1},\cdots, e_n, \pi e_{h+1}, \cdots, \pi e_n\}\subset \Lambda_{h})
\end{array}
\end{equation}
of rank $n+t+h$ (resp. rank $2n-(h-t)$). The dual of $L_\pm$ with respect to $( ~ , ~)$ is
\begin{equation}
\begin{array}{c}
L_+^\bot=\text{span}_{R}\{\pi e_{h+1}, \cdots, \pi e_{n-t}\}\subset \Lambda_{h},   \\
(\text{resp.}~ L_-^\bot=\text{span}_{R}\{\pi e_{n-h+1}, \cdots, \pi e_{n-t}\}\subset \Lambda_{n-h}).
\end{array}
\end{equation}
of rank $n-t-h$ (resp. $h-t$). With respect to the reordered basis, the condition $\calG_{h}\subset L_+^\bot, ~\calG_{n-h}\subset L_-^\bot$ translates to
\begin{equation}\label{relation-Y}
\begin{array}{c}
T_{1,2}'=T_{1,3}'=0,\quad T_{1,4}'=0;  \\
V_{1,2}'=V_{1,3}'=0,\quad V_{1,4}'=0,\quad V_2=0.
\end{array}
\end{equation}
The invertible element $t_{i_0}$ occurs in $1\leq i_0\leq h-t$ or $2h+1\leq i_0\leq n$, and $v_{j_0}\in V_{1,1}'$ with $1\leq j_0\leq h-t$. We need to consider the affine charts $\U_{i_0, j_0}$ in case (1), (2) of Proposition \ref{affinechartU}. Similar to Proposition \ref{prop X-strata}, we set the open affine chart $U_{i_0, j_0}:=\calU_{i_0, j_0}\cap \M^{\nspl, [2h]}_n(2t)$ and $U_{i_0, j_0}':=\calU_{i_0, j_0}\cap \calM^{ [2h]}_n(2t)$ for $t<h$.

\begin{Proposition}\label{prop Y-strata}
(1) When $2h+1\leq i_0\leq n$, $1\leq j_0\leq h-t$, the affine chart $U_{i_0, j_0}\subset \M^{\nspl, [2h]}_n(2t)$ is isomorphic to
\[
\Spec \frac{k[V_{1,1}', T_2, t_{j_0}, z_{i_0^\vee}]}{(t_{i_0}-1, v_{j_0}-1, z_{i_0^\vee}(z_{i_0^\vee} (T_2^tHT_2)-2t_{j_0}))}.
\]
It has 2 irreducible components of equidimension $n-h-t-1$. The affine chart $U_{i_0, j_0}'\subset \calM^{ [2h]}_n(2t)$ is smooth of dimension $n-h-t-1$ and isomorphic to
\[
\Spec \frac{k[V_{1,1}', T_2, z_{i_0^\vee}]}{(t_{i_0}-1, v_{j_0}-1)}\simeq \mathbb{A}_k^{n-h-t-1}.
\]
(2) When $1\leq i_0, j_0 \leq h-t$,  the affine chart $U_{i_0, j_0}\subset \M^{\nspl, [2h]}_n(2t)$ is isomorphic to
\[
\Spec \frac{k[T_{1,1}', T_2, s_{i_0^\vee}, t_{j_0}^{-1}]}{(t_{i_0}-1,  s_{i_0^\vee}(s_{i_0^\vee}(T_2^tHT_2)-2) )}.
\]
It has 2 irreducible components of equidimension $n-h-t-1$. The affine chart $U_{i_0, j_0}'\subset \calM^{ [2h]}_n(2t)$ is smooth of dimension $n-h-t-1$ and isomorphic to
\[
\Spec \frac{k[T_{1,1}', T_2, s_{i_0^\vee}, t_{j_0}^{-1}]}{(t_{i_0}-1,  s_{i_0^\vee}(T_2^tHT_2)-2 )}.
\]
\end{Proposition}

\begin{proof}
When $2h+1\leq i_0\leq n$, $1\leq j_0\leq h-t$, note that $v_{i_0}=0\in V_2$. Then we get (1) from Propositions \ref{affinechartU} and \ref{prop splflat} (Case 1). (2) follows directly from Proposition \ref{affinechartU}, \ref{prop splflat} (Case 2) and relation (\ref{relation-Y}). 
\end{proof}

\subsubsection{Intersection of $\calZ$-strata and $\calY$-strata models}\label{intersection}
Let $\Lambda_1\subset F^n$ be a vertex lattice of type $2t_1$ with $t_1> h$, and $\Lambda_2^\sharp\subset F^n$ be a vertex lattice of type $2t_2$ with $t_2< h$. Consider the intersection $\M^{\nspl, [2h]}_n (2t_1)\cap \M^{\nspl, [2h]}_n (2t_2)$ and $\calM^{[2h]}_n (2t_1)\cap \calM^{[2h]}_n (2t_2)$.

By Propositions \ref{prop X-strata} and \ref{prop Y-strata}, it is easy to see that the only non-zero parts in $V, T$ are $V_{1,1}'$ and $T_{2,1}$. Let $U_{i_0, j_0}$ (resp. $U_{i_0, j_0}'$) be the affine chart in $\M^{\nspl, [2h]}_n (2t_1)\cap \M^{\nspl, [2h]}_n (2t_2)$ (resp. $\calM^{[2h]}_n (2t_1)\cap \calM^{[2h]}_n (2t_2)$) with $2h+1\leq i_0\leq t_1+h$, $1\leq j_0\leq h-t_2$. We have

\begin{Proposition}
The affine chart $U_{i_0, j_0}$ is isomorphic to $U_{i_0, j_0}'$. Both are smooth and isomorphic to $\mathbb{A}_k^{t_1-t_2-1}$.   
\end{Proposition}

\begin{proof}
The statement follows from Propositions \ref{prop X-strata} and \ref{prop Y-strata}. Note that in $U_{i_0, j_0}$ we have $T_2^tHT_2=0$ by $T_{2,2}=0, T_{2,3}=0$, and $v_{i_0}=0$ by $V_2=0$. Hence, it follows that
\[
U_{i_0, j_0}\simeq U_{i_0, j_0}'\simeq \Spec \frac{k[V_{1,1}',  T_{2,1}, z_{i_0^\vee}]}{(t_{i_0}-1, v_{j_0}-1)}= \mathbb{A}^{t_1-t_2-1}_k.
\]
\end{proof}
Combining the above with Propositions \ref{prop X-strata} and \ref{prop Y-strata}, we obtain Theorem \ref{thm splstrata}.

\section{Blow-up models}\label{BLupsection}
The main result of this section is that, for $t\neq h$, the strata model $\calM^{[2h]}_n(2t)$ is the blow-up of the strata local model $\M^{\loc,[2h]}_n(2t)$ at the worst point $*$ (see~§\ref{Sec. 4.3}). We denote this blow-up by $\M^{\bl,[2h]}_n(2t)$. By \cite[Proposition 2.5.1]{HLS1}, it suffices to compute the blow-up of an affine neighborhood $\calU \subset \M^{\loc,[2h]}_n(2t)$ containing $*$. As in §\ref{affinechart strata}, we treat the cases $t>h$ and $t<h$ separately.

\subsection{Affine chart for the blow-up $t>h$}
Let $\Lambda\subset C$ be a vertex lattice of type $2t$ with $t> h$. We begin by computing the blow-up $\M^{\bl,[2h]}_n(2t)$. By an unramified extension, we can reduce to the case where the hermitian form is split. Recall from \cite[\S 4.5.1]{HLS2} that an affine chart $\calU \subset \M^{\loc, [2h]}_n(2t)$ is isomorphic to
\begin{equation}
\calU \simeq  \Spec A=\Spec \frac{k[X,Y]}{(X-X^{ad}, \wedge^2(X,Y))}.
\end{equation}
Here $X^{ad}:=HX^tH$ where $H$ is the unit anti-diagonal matrix and the matrices $X, \, Y$ are of sizes $(t-h)\times (t-h)$ and $(t-h)\times 2h$ respectively.

Consider the intersection of $\calU$ and the worst point $\calU \cap (\Pi\Lambda_{-h}, \Pi \Lambda_{h})$. This smooth closed subscheme is defined by the ideal $I = (X, Y)$. Let $\calU^\bl$ be the blow-up of $\calU$ along $I=(X,Y)$. We have the blow-up morphism $\rho: \U^{\rm bl}\rightarrow \U$, where
\begin{equation}
	\U^{\rm bl}:={\rm Proj (\tilde{A})}.
\end{equation}
Here $\tilde{A}$ is the graded $A$-algebra $\oplus_{d\geq 0} I^d$ with $I^0=A$.

\begin{Proposition}\label{prop Zbl}
The blow-up $\U^{\rm bl}$ is smooth of dimension $t+h$. The exceptional divisor $\calU^{\rm exc}$ is isomorphic to
\[
\PP_k^{t-h-1}\times \PP_k^{2h}.
\]
\end{Proposition}

\begin{proof}
There are $(t-h)(t+h)$- affine patches in the blow-up of $\Spec A$ along the ideal $I=(X, Y)$. Let
\[
(X, Y)=\left(\begin{array}{ccc|ccc}
 x_{1,1}    &  \cdots & x_{1,t-h}& x_{1, t-h+1}& \cdots &x_{1, t+h}\\
 \vdots& \ddots& \vdots& \vdots&\ddots&\vdots\\
x_{t-h,1}    &  \cdots & x_{t-h,t-h}& x_{t-h, t-h+1}& \cdots &x_{t-h, t+h}\\
\end{array}
\right).
\]
Let $\varepsilon=(\epsilon_{i,j})_{1\leq i\leq t-h, 1\leq j\leq t+h}$ be the represented elements of $x_{i,j}$ in $(X, Y)$, considered as homogeneous elements of degree $1$. Assume that $\varepsilon_{r,s}=1$. The affine chart $D_+(\varepsilon_{r,s})\subset \U^{\rm bl}$ is  given by 
\begin{equation}
\Spec \frac{A[\varepsilon_{i,j}]_{1\leq i\leq t-h, 1\leq j\leq t+h}}{J_{r,s}},	
\end{equation}
where $J_{r,s}$ is genetrated by $\varepsilon_{r,s}-1$ and
\begin{equation}
\{f\in A[\varepsilon_{i,j}]\mid \exists N\geq 0, x_{r,s}^N\cdot f\in (x_{i,j}-x_{r,s}\varepsilon_{i,j})_{1\leq i\leq t-h, 1\leq j\leq t+h}\}.	
\end{equation}
Thus, the matrix $(X, Y)=x_{r,s} \varepsilon$, and $\wedge^2(X,Y)$ is equivalent to $x_{r,s}^2(\varepsilon_{i,j}-\varepsilon_{r,j}\varepsilon_{i,s})=0$. Hence, $(\varepsilon_{i,j}-\varepsilon_{r,j}\varepsilon_{i,s})\in J_{r,s}$. The condition $X=X^{ad}$ is equivalent to $x_{i,j}=x_{t-h+1-j,t-h+1-i}$ for all $1\leq i, j\leq t-h$. Note that $x_{i,j}=x_{r,s}\varepsilon_{i,j}=x_{r,s}\varepsilon_{r,j}\varepsilon_{i,s}$. Thus,
\[
x_{r,s}(\varepsilon_{r,j}\varepsilon_{i,s}-\varepsilon_{r,t-h+1-i}\varepsilon_{t-h+1-j,s})=0.
\]
Let $i=r$, we get $x_{r,s}(\varepsilon_{r,j}-\varepsilon_{r,t-h+1-r}\varepsilon_{t-h+1-j,s})=0$, so that 
\begin{equation}\label{eq Jrs}
(\varepsilon_{r,j}-\varepsilon_{r,t-h+1-r}\varepsilon_{t-h+1-j,s})\in J_{r,s} ~\text{for all}~ 1\leq j\leq t-h.   
\end{equation}
If $r+s=t-h+1$, we obtain $\varepsilon_{r,t-h+1-r}=\varepsilon_{r,s}=1$. It is easy to see that 
\[
D_+(\varepsilon_{r,s})=\Spec \frac{k[\varepsilon_{1,s}, \cdots, \varepsilon_{t-h,s}, \varepsilon_{r,t-h+1},\cdots, \varepsilon_{r,t+h}, x_{r,s}]}{(\varepsilon_{r,s}-1)}\simeq \mathbb{A}^{t+h}_k.
\]
If $r+s\neq t-h+1$, there are two cases. We either have $s> t-h$ or $1\leq t\leq t-h$. In the first case, the variable $\varepsilon_{r,s}$ occurs in the sequences $\varepsilon_{1,s}, \cdots, \varepsilon_{t-h,s}$ and $\varepsilon_{r,t-h+1},\cdots, \varepsilon_{r,t+h}$. Thus,
\[
D_+(\varepsilon_{r,s})=\Spec \frac{k[\varepsilon_{i,s}, \varepsilon_{r,j}, \varepsilon_{r, t-h+1-r}, x_{r,s}]}{(\varepsilon_{r,s}-1)}\simeq \mathbb{A}^{t+h}_k.
\]
Here $i\in \{1, \cdots, t-h\}, j\in \{t-h+1, \cdots, t+h\}-\{s\}$. In the second case, by letting $j=s$ in (\ref{eq Jrs}), we get $1=\varepsilon_{r,t-h+1-r}\varepsilon_{t-h+1-s,s}$, and
\[
D_+(\varepsilon_{r,s})=\Spec \frac{k[\varepsilon_{1,s}, \cdots, \varepsilon_{t-h,s}, \varepsilon_{r,t-h+1},\cdots, \varepsilon_{r,t+h},\varepsilon_{r,t-h+1-r}, x_{r,s}]}{(\varepsilon_{r,s}-1, \varepsilon_{r,t-h+1-r}\varepsilon_{t-h+1-s,s}-1)}.
\]
From the above, we prove that all affine patches $D_+(\varepsilon_{r,s})$ are smooth of dimension $t+h$, so as $\calU^\bl$. Finally, consider the exceptional divisor $\calU^{\rm exc}$, where we set $x_{r,s}=0$. We have $\calU^{\rm exc}=\PP_k^{t-h-1}\times \PP_k^{2h}$.
\end{proof}

\subsection{Affine chart for the blow-up $t<h$}
Now we consider the blow-up $\M^{\bl,[2h]}_n(2t)$ when $t<h$. Let $\Lambda\subset C$ be a vertex lattice of type $2t$ with $t< h$. Recall from \cite[\S 4.5.2]{HLS2} that an affine chart $\calU \subset \M^{\loc, [2h]}_n(2t)$ is isomorphic to
\[
\calU \simeq  \Spec B=\Spec \frac{k[Z]}{( \wedge^2(Z))},
\]
where $Z$ is of size $(n-2h)\times (h-t)$. The worst point $(\Pi \Lambda_{n-h}, \Pi \Lambda_h)$ is defined by the ideal $I=(Z)$. Let $\calU^\bl$ be the blow-up of $\calU$ along $I$. We have
\begin{Proposition}\label{prop Ybl}
The blow-up $\U^{\rm bl}$ is smooth of dimension $n-h-t-1$. The exceptional divisor $\calU^{\rm exc}$ is isomorphic to
\[
\PP_k^{h-t-1}\times \PP_k^{n-2h-1}.
\]
\end{Proposition}

\begin{proof}
The proof is similar to Proposition \ref{prop Zbl}. There are $(n-2h)(h-t)$- affine patches in the blow-up of $\Spec B$ along the ideal $I=(Z)$. Let $Z=(z_{i, j})$ and let $\varepsilon=(\varepsilon_{i,j})$ be the represented elements of $z_{i,j}$ for all $1\leq i\leq n-2h, 1\leq j\leq h-t$. Assume that $\varepsilon_{r,s}=1$. The affine chart $D_+(\varepsilon_{r,s})\subset \U^{\rm bl}$ is  given by 
\begin{equation}
\Spec \frac{B[\varepsilon_{i,j}]_{1\leq i\leq n-2h, 1\leq j\leq h-t}}{J_{r,s}},	
\end{equation}
where $J_{r,s}$ is genetrated by $\varepsilon_{r,s}-1$ and
\begin{equation}
\{f\in A[\varepsilon_{i,j}]\mid \exists N\geq 0, z_{r,s}^N\cdot f\in (z_{i,j}-z_{r,s}\varepsilon_{i,j})_{1\leq i\leq n-2h, 1\leq j\leq h-t}\}.	
\end{equation}
Thus, the matrix $Z=z_{r,s} \varepsilon$, and $\wedge^2(Z)$ is equivalent to $z_{r,s}^2(\varepsilon_{i,j}-\varepsilon_{r,j}\varepsilon_{i,s})=0$. so that $(\varepsilon_{i,j}-\varepsilon_{r,j}\varepsilon_{i,s})\in J_{r,s}$. Let $a_i=\varepsilon_{i,s}, b_j=\varepsilon_{r,j}$. Note that $a_r=b-s=\varepsilon_{r,s}=1$. Therefore, we obtain the affine patch
\[
D_+(\varepsilon_{r,s})=\Spec \frac{k[a_i, b_j,  z_{r,s}]}{(a_r-1, b_s-1)}\simeq \mathbb{A}^{n-h-t-1}_k,
\]
for $1\leq i\leq n-2h, 1\leq j\leq h-t$. It is easy to see that when $z_{r,s}=0$, the exceptional divisor $\calU^{\rm exc}$ is isomorphic to
$\PP_k^{h-t-1}\times \PP_k^{n-2h-1}$.
\end{proof}

\subsection{Relation of strata models and the blow-up}\label{Sec. 4.3}
 
In this subsection we prove that, for \(t\neq h\), the strata model $\calM^{[2h]}_n(2t)$ coincides with the blow-up $\M^{\bl,[2h]}_n(2t)$. The main result is the following proposition.

\begin{Proposition}\label{prop 43}
Assume that $t\neq h$. The strata models $\calM^{[2h]}_n(2t)$ are equal to the blow-up $\M^{\bl,[2h]}_n(2t)$  . 
\end{Proposition}
\begin{proof}
From Theorem \ref{thm 3.6}, the forgetful morphism $\tau : \M^{\spl,[2h]}_n \rightarrow \M^{\loc,[2h]}_n$ is the blow-up $ \M^{\bl,[2h]}_n$ of $\M^{\loc, [2h]}_n$ along the worst point $*$. Thus, we have 
\[
\M^{\bl,[2h]}_n =  \M^{\spl,[2h]}_n \subset  \M^{\nspl,[2h]}_n \xrightarrow{\tau}  \M^{\loc,[2h]}_n.
\]
Moreover, the map $\tau$ induces an isomorphism 
\[
\M^{\nspl,[2h]}_n\setminus \tau^{-1}(*) \cong  \M^{\loc,[2h]}_n \setminus \{*\}.
\]
Recall the definitions of the strata models: $\calM^{[2h]}_n(2t)=\M^{\spl, [2h]}_n \cap \M^{\nspl,[2h]}_n(2t)$ and $\M^{\bl,[2h]}_n(2t)$ is the strict transform of $\M^{\loc,[2h]}_n(2t)$. Set $D: =\M^{\loc,[2h]}_n(2t) \cap \{*\}$ and let $U=\tau^{-1} \bigl(\M^{\loc,[2h]}_n(2t)\setminus D\bigr).$ By definition $\M^{\bl,[2h]}_n(2t)$ is the scheme-theoretic closure of $U$ in $\M^{\bl,[2h]}_n$. Observe that $ U \subset \calM^{[2h]}_n(2t)$, since $\tau$ is an isomorphism over $\M^{\loc,[2h]}_n\setminus\{*\}$ and thus
\[
\begin{aligned}
U
&= \tau^{-1}\bigl(\M^{\loc,[2h]}_n(2t)\setminus D\bigr)
 = \M^{\nspl,[2h]}_n(2t)\cap \tau^{-1}\!\bigl(\M^{\loc,[2h]}_n\setminus\{*\}\bigr)\\
&\subset \M^{\spl,[2h]}_n \cap \M^{\nspl,[2h]}_n(2t)
 = \calM^{[2h]}_n(2t).
\end{aligned}
\]
Hence, $U$ is an open (hence dense) subset of the irreducible closed subscheme $ \calM^{[2h]}_n(2t)$ of $\M^{\spl,[2h]}_n= \M^{\bl,[2h]}_n $. Thus from the above we deduce that $\M^{\bl,[2h]}_n(2t)=\calM^{[2h]}_n(2t)$.
\end{proof}
\quash{
\begin{Corollary}\label{Cor.321}
The BT strata $\mathcal{Z}^{\rm bl} (\Lambda),\, \mathcal{Y}^{\rm bl}(\Lambda^\sharp) $ are equal to $\mathcal{Z}^{\rm spl} (\Lambda)$ and $ \mathcal{Y}^{\rm spl}(\Lambda^\sharp) $ respectively.
\end{Corollary}
}
\section{Properties of Bruhat-Tits strata}\label{LPBT}
In this section, the goal is to obtain certain nice properties (e.g. reducedness, smoothness, irreducibility) for the Bruhat-Tits strata $\mathcal{Z} (\Lambda),\, \mathcal{Y}(\Lambda^\sharp)$ defined in (\ref{eq ZYstrata}). To do this, we will relate these Bruhat-Tits strata with the strata models via the local model diagram.


First, let us briefly recall the construction of such a local diagram for the BT-strata $\mathcal{Z}^{\rm loc} (\Lambda),\, \mathcal{Y}^{\rm loc}(\Lambda^\sharp) $ given in \cite[\S 4]{HLS2}. Assume that $\Lambda\subset C$ is a vertex lattice of type $2t$. 

\begin{enumerate}
\item For $t \geq h$, define $\tilde{\mathcal{Z}}^{\rm loc}(\Lambda)$ to be a projective formal scheme over $\bar{k}$ that represents the functor sending each $\bar{k}$-algebra $R$ to the set of tuples $(X, \iota, \lambda, \rho, f)$ where:
\begin{itemize}
    \item $(X, \iota, \lambda, \rho) \in \mathcal{Z}^{\rm loc} (\Lambda)(R)$, 
    \item $f$ is an isomorphism between the standard lattice chain $\calL_{[2h,2t],R} := \calL_{[2h,2t]}\otimes R$ and the lattice chain of de Rham realizations:
    \[
f: \begin{tikzcd}
\Lambda_{-t,R}\arrow{r} \arrow[d,"\sim"]
&
\Lambda_{-h,R} \arrow{r} \arrow[d,"\sim"]
&
\Lambda_{h,R} \arrow{r} \arrow[d,"\sim"]
&
\Lambda_{t,R} \arrow[d,"\sim"]
\\
D(X_{\Lambda}) \arrow{r}
&
D(X) \arrow{r} \arrow{r} 
&
D(X^{\vee})\arrow{r}
&
D(X_{\Lambda^{\sharp}})
\end{tikzcd}.
\]
\end{itemize}

\item For $t \leq h$, define $\tilde{\mathcal{Y}}^{\rm loc}(\Lambda^{\sharp})$ to be a projective formal scheme over $\bar{k}$ that represents the functor sending each $\bar{k}$-algebra $R$ to the set of tuples $(X, \iota, \lambda, \rho, f)$ where:
\begin{itemize}
    \item $(X, \iota, \lambda, \rho) \in \mathcal{Y}^{\rm loc} (\Lambda^{\sharp})(R)$, 
    \item $f$ is an isomorphism between the standard lattice chain $\calL_{[2h,2t],R} $ and the lattice chain of de Rham realizations:
    \[
f: \begin{tikzcd}
\Lambda_{t,R}\arrow{r} \arrow[d,"\sim"]
&
\Lambda_{h,R} \arrow{r} \arrow[d,"\sim"]
&
\Lambda_{n-h,R} \arrow{r} \arrow[d,"\sim"]
&
\Lambda_{n-t,R} \arrow[d,"\sim"]
\\
D(X_{\Lambda^\sharp}) \arrow{r}
&
D(X^{\vee}) \arrow{r} \arrow{r} 
&
D(X)\arrow{r}
&
D(X_{\pi^{-1}\Lambda})
\end{tikzcd}.
\]
\end{itemize}
\end{enumerate}
Recall that $\scrG_{[2h, 2t]}$ is the smooth group scheme of automorphisms of the lattice chain $ \calL_{[2h,2t]}$. We have the local model diagram 
 \begin{equation}\label{LMdiagram1}
\begin{tikzcd}
&\tilde{\mathcal{Z}}^{\rm loc}(\Lambda)\arrow[dl, "\psi_1"']\arrow[dr, "\psi_2"]  & \\
\mathcal{Z}^{\rm loc}(\Lambda)  &&  \M^{\loc, [2h]}_n(2t)
\end{tikzcd}
\end{equation}
where $\psi_1$ is a smooth $\scrG_{[2h, 2t], \bar{k}}$-torsor of relative dimension $\operatorname{dim}\scrG_{[2h, 2t], \bar{k}}$ and $\psi_2$ is a smooth morphism of relative dimension $\operatorname{dim}\scrG_{[2h, 2t], \bar{k}}$. (We get similar local model diagrams for $\mathcal{Y}^{\rm loc}(\Lambda^{\sharp})$ and $\calZ^\loc(\Lambda_1)\cap \mathcal{Y}^{\rm loc}(\Lambda^{\sharp}_2)$.) Here, $\psi_1$ is defined by forgetting the trivialization $f$ and $ \psi_2$ is defined by attaching the Hodge filtration of the strict $O_{F_0}$-modules to the lattice chain through the isomorphism $f$ (see \cite[\S 4.2]{HLS2} for more details). 

Recall from \S \ref{StrataBlModels} there exists a projective $\scrG_{[2h]}$-morphism 
\[
\tau: \M^{\nspl,[2h]}_n(2t)\rightarrow \M^{\loc,[2h]}_n(2t).
\]
From the above, we deduce that the stratum $\mathcal{Z}^\Kra(\Lambda)$ in the Kr\"amer RZ space is a linear modification of $\mathcal{Z}^{\rm loc}(\Lambda)$ in the sense of \cite[\S 2]{P}. In particular there is a local model diagram 
 \begin{equation}\label{LMdiagram3}
\begin{tikzcd}
&\tilde{\mathcal{Z}}^\Kra(\Lambda)\arrow[dl, "\psi'_1"']\arrow[dr, "\psi'_2"]  & \\
\mathcal{Z}^\Kra(\Lambda)  &&  \M^{\nspl, [2h]}_n(2t)
\end{tikzcd}.
\end{equation}
We have similar local model diagrams for $\mathcal{Y}^{\Kra}(\Lambda^{\sharp})$ and $\calZ^\Kra(\Lambda_1)\cap \mathcal{Y}^{\Kra}(\Lambda^{\sharp}_2)$. 

It is worth mentioning that for $S=\Spec R$, with $R$ a $\bar{k}$-algebra, the condition $ x_{*} (\operatorname{Lie}(Y\times S)) \subset \mathcal{F}_X\cap \calF_{X^\vee}$ of $\calZ^{\Kra}(\Lambda)(S)$ (see Definiton \ref{def 2.10}) is equivalent, via the local model diagram, to $\calG_{\pm h}\subset L_\pm^\bot$ in the strata model $\calM^{[2h]}_n(2t)(R)$ for $t>h$. Note that 
$ x_{*} (\operatorname{Lie}(Y\times S)) \subset \mathcal{F}_X\cap \calF_{X^\vee}$ translates to 
\begin{equation}\label{eq M'M''}
\Lambda \otimes W_{O_{F_0}}(\kappa) \subset M'(X) \subset M(X)^\sharp,\quad
    \Lambda \otimes W_{O_{F_0}}(\kappa) \subset M''(X) \subset M(X)   
\end{equation}
by Proposition \ref{DieudLatt}. Recall that there are perfect pairings
\[
\operatorname{Fil} (X) \times \operatorname{Lie}(X^\vee)\to \mathcal{O}_S \text{  and  } \operatorname{Fil} (X^\vee) \times \operatorname{Lie}(X)\to \mathcal{O}_S
\]
induced by (\ref{perfectpair}). Let $\calF_{X^\vee}^\bot\subset \operatorname{Fil} (X)$, $\calF_{X}^\bot\subset \operatorname{Fil} (X^\vee)$ be the perpendicular complement of $\calF_{X^\vee}\subset \operatorname{Lie}(X^\vee)$, $\calF_{X}\subset \operatorname{Lie}(X)$ respectively. Since $D(X)=M(X)/\pi_0 M(X)$ (resp. $D(X^\vee)=M(X^\vee)/\pi_0 M(X^\vee)$) can be identified with the standard lattice $\Lambda_{-h, R}$  (resp. $\Lambda_{h, R}$), we translate $\calF_{X^\vee}^\bot\subset \operatorname{Fil} (X)$ (resp. $\calF_{X}^\bot\subset \operatorname{Fil} (X^\vee)$) to $\calG_{-h}\subset \calF_{-h}$ (resp. $\calG_{h}\subset \calF_{h}$), where $\calG_{-h}$ (resp. $\calG_{h}$) is the lattice corresponding to $\calF_{X^\vee}^\bot$ (resp. $\calF_{X}^\bot$). Then note that $M'(X)={\rm Pr}_2^{-1}(\calF_{X^\vee})$, $M''(X)={\rm Pr}_1^{-1}(\calF_{X})$, where
\[
{\rm Pr}_1: M(X)\rightarrow \operatorname{Lie}(X),\quad 
{\rm Pr}_2: M(X^\vee)\rightarrow \operatorname{Lie}(X^\vee).
\]
Therefore, the relation (\ref{eq M'M''}) is equivalent to $L_\pm \subset \calG_{\pm h}^\bot$, and so $\calG_{\pm h}\subset \Lambda_\pm ^\bot$. A similar argument works for the $\calY^\Kra$-strata.

Let $\Lambda_0\subset C$ be a vertex of type $2h$. Recall that we have morphisms $\tau_1: \mathcal{Z}^{\rm Kra}(\Lambda_0)\rightarrow \mathcal{Z}^{\rm loc}(\Lambda_0)$ (resp. $\tau_2: \mathcal{Y}^{\rm Kra}(\Lambda_0^\sharp)\rightarrow \mathcal{Y}^{\rm loc}(\Lambda_0^\sharp)$) given by $(X, \iota, \Lambda_0, \rho, \mathcal{F}_X, \mathcal{F}_{X^\vee})\mapsto (X, \iota, \Lambda_0, \rho)$, where $ \mathcal{Z}^{\rm loc}(\Lambda_0)(\kappa)=\mathcal{Y}^{\rm loc}(\Lambda_0^{\sharp})(\kappa)=\{\Lambda_0 \otimes W_{O_{F_0 }}(\kappa) \}$ (the worst point); see also \cite[Corollary 2.11]{ZacZhao3}. 

\begin{Proposition}\label{prop KraNexc}
Assume $h$ is not $\pi$-modular. The strata $ \mathcal{Z}^{\rm Kra}(\Lambda_0)(\kappa)=\mathcal{Y}^{\rm Kra}(\Lambda_0^\sharp)(\kappa)$ is a linked (quiver) variety with dimension $n-1$. It has two irreducible components and their intersection is irreducible of dimension $n-2$.
\end{Proposition}

\begin{proof}
From the local model diagram \ref{LMdiagramKra}, each point of $ \mathcal{Z}^\Kra(\Lambda_0)$ and $\mathcal{Y}^{\rm Kra}(\Lambda_0^\sharp)$ is \'etale locally isomorphic to a point in the naive exceptional divisor $\NExc$. Now, the result follows from \cite[Proposition 2.5.4]{HLS1}.
\end{proof}

The main result in this section is the following. 

\begin{Theorem}\label{Reducedness}
Let $\Lambda_1 $ (resp. $\Lambda_2$) be a vertex lattice of type $2t_1$, (resp. $2t_2$) with $t_2<h<t_1$. We have

a) The moduli functor $ \mathcal{Z}(\Lambda_1)$ is smooth, irreducible and of dimension $t+h$.

b) The moduli functor $\mathcal{Y}(\Lambda_2^{\sharp}) $ is smooth, irreducible and of dimension $ n-t-h-1$.

c) The moduli functor $\calZ(\Lambda_1)\cap \mathcal{Y}(\Lambda^{\sharp}_2)$ is smooth, irreducible and of dimension $ t_1-t_2-1$.
\end{Theorem}
\begin{proof}

By restricting the morphism $\tau$ to the strata model, we have $\tau: \calM^{[2h]}_n(2t)\rightarrow \M^{\loc,[2h]}_n(2t)$, where $\calM^{[2h]}_n(2t)$ is the blow-up of $\M^{\loc,[2h]}_n(2t)$ along the worst point $*$ by Proposition \ref{prop 43}. From the above, we deduce that $\mathcal{Z}(\Lambda_1)$ is a linear modification of $\mathcal{Z}^{\rm loc}(\Lambda_1)$ in the sense of \cite[\S 2]{P} and in particular there is a local model diagram 
 \begin{equation}\label{LMdiagram2}
\begin{tikzcd}
&\tilde{\mathcal{Z}}(\Lambda_1)\arrow[dl, "\psi'_1"']\arrow[dr, "\psi'_2"]  & \\
\mathcal{Z}(\Lambda_1)  &&  \calM^{ [2h]}_n(2t)
\end{tikzcd}.
\end{equation}
From the local model diagram we have that every point of $ \mathcal{Z}(\Lambda_1)$ has an \'etale neighborhood which is also \'etale over the strata model $  \calM^{[2h]}_n(2t)$. Now smoothness and the dimension formula for $ \mathcal{Z}(\Lambda_1)$ follow from Theorem \ref{thm splstrata}. 

Next, recall from \S \ref{BT_strata} that $\mathcal{Z}^{\rm loc}(\Lambda_1)$ and $\mathcal{Z}(\Lambda_1)$ are representable by projective schemes over $\bar{k}$. Since blowing-up commutes with \'etale localization, $\mathcal{Z}(\Lambda_1) $ is the blow-up of $\mathcal{Z}^{\rm loc}(\Lambda_1)$. By \cite{HLS2}, the BT-strata $\mathcal{Z}^{\rm loc}(\Lambda_1)$ is an integral scheme, so that its blow-up $\mathcal{Z}(\Lambda_1)$ is also integral (therefore irreducible). The same argument applies to $\mathcal{Y}(\Lambda_2^{\sharp}) $ and $\calZ(\Lambda_1)\cap \mathcal{Y}(\Lambda^{\sharp}_2)$.
\end{proof}

\quash{
\begin{Proposition}\label{prop KraNexc}
Assume $h$ is not $\pi$-modular. The strata $ \mathcal{Z}(\Lambda_0)(\kappa)=\mathcal{Y}(\Lambda_0^\sharp)(\kappa)$ is the blow-up of $\PP^{n-1}_k$ along $\PP^{2h-1}_k$.
\end{Proposition}
}

\section{Bruhat-Tits stratification}\label{BTstrat.}
In this section, we will define the Bruhat-Tits stratification of the reduced subscheme $\calN_{n, {\rm red}}^\spl$ (the \textit{reduced basic locus}) of the special fiber $\overline{\mathcal{N}}^\spl_n$. 

Let $\kappa$ be any perfect field over $\bar{k}$. Recall that $M = M(X)$ the Dieudonn\'e module of $(X, \iota, \lambda, \rho)\in \calN_n$. We denote by $T_i(M)$ (resp. $T_i(M^\sharp)$) the summation $M+\tau(M)+\cdots +\tau^i(M)$ (resp. $M^\sharp+\tau(M^\sharp)+\cdots +\tau^i(M^\sharp)$). By \cite[Proposition 2.17]{RZbook}, there exists a smallest nonnegative integer $c$ (resp. $d$) such that $T_c(M)$ (resp. $T_d(M^\sharp)$) is $\tau$-invariant. Set $\Lambda_1=T_d(M^\sharp)^\sharp\cap C$, $\Lambda_2=T_c(M)\cap C$. 

\begin{Proposition}\label{prop 51} 
We have $\Lambda_1\otimes W_{O_{F_0}}(\kappa)\subset M\subset \Lambda_2\otimes W_{O_{F_0}}(\kappa)$, and the $W_{O_{F_0}}(\kappa)$-lattices $M$ satisfy one of the following:
\begin{itemize}
    \item (Case $\calZ$)\quad $\Lambda_1\subset C$ is a vertex lattice of type $2t_1\geq 2h$ with 
    \[
    \pi M^\sharp\subset
    \pi \Lambda_1^\sharp\otimes W_{O_{F_0}}(\kappa)\subset
    \Lambda_1\otimes W_{O_{F_0}}(\kappa)\subset
    M \subset
    M^\sharp \subset
    \Lambda_1^\sharp\otimes W_{O_{F_0}}(\kappa),
    \]
 and $\Lambda_1$ is the maximal vertex lattice in $C$ such that $\Lambda_1\otimes W_{O_{F_0}}(\kappa)$ is contained in $M$.
    \item (Case $\calY$)\quad $\Lambda_2\subset C$ is a vertex lattice of type  $2t_2\leq 2h$ with
    \[
    \pi \Lambda_2^\sharp\otimes W_{O_{F_0}}(\kappa)\subset
    \pi M^\sharp\subset
    M\subset
    \Lambda_2\otimes W_{O_{F_0}}(\kappa) \subset
    \Lambda_2^\sharp\otimes W_{O_{F_0}}(\kappa) \subset
    M^\sharp,
    \]
    and $\Lambda_2$ is the minimal vertex lattice in $C$ such that $\Lambda_2\otimes W_{O_{F_0}}(\kappa)$ contains $M$.
\end{itemize}
\end{Proposition}

\begin{proof}
See \cite[Proposition 5.3]{HLS2}.
\end{proof}

Recall that $L_{\mathcal{Z}}$ (resp. $L_{\mathcal{Y}}$) denotes the set of all vertex lattices in $C$ of type $\geq 2h$ (resp. $\leq 2h$). Let's review  the Bruhat-Tits stratification of the reduced subscheme $\calN_{n, {\rm red}}$ of the special fiber $\overline{\mathcal{N}}_n$.

\begin{Theorem}\label{BTStratLoc}
The reduced locus $\calN_{n, {\rm red}}$ is the union of closed subvarieties:
\begin{equation}\label{BT_strataLoc}
    \calN_{n, {\rm red}} = \left( \bigcup_{\Lambda_1 \in L_{\mathcal{Z}} }\mathcal{Z}^{\rm loc}(\Lambda_1) \right) \cup \left( \bigcup_{\Lambda_2 \in L_{\mathcal{Y}}} \mathcal{Y}^{\rm loc}(\Lambda_2^\sharp) \right).
\end{equation}
These strata satisfy the following inclusion relations:
    \begin{enumerate}
        \item For any \(\Lambda_1, \Lambda'_1 \in L_{\mathcal{Z}}\), \(\mathcal{Z}^{\rm loc}(\Lambda_1) \subseteq \mathcal{Z}^{\rm loc}(\Lambda'_1)\) if and only if \(\Lambda_1 \supseteq \Lambda'_1\),
        \item For any \(\Lambda_2, \Lambda'_2 \in L_{\mathcal{Y}}\), \(\mathcal{Y}^{\rm loc}(\Lambda_2^\sharp) \subseteq \mathcal{Y}^{\rm loc}((\Lambda'_2)^\sharp)\) if and only if \(\Lambda_2 \subseteq \Lambda'_2\).
    \item For any \(\Lambda_1 \in L_{\mathcal{Z}}\) and \(\Lambda_2 \in L_{\mathcal{Y}}\), the intersection \(\mathcal{Z}^{\rm loc}(\Lambda_1) \cap \mathcal{Y}^{\rm loc}(\Lambda_2^\sharp)\) is non-empty if and only if \(\Lambda_1 \subseteq \Lambda_2\).
    \item For \(\Lambda_0 \in L_{\mathcal{Z}} \cap L_{\mathcal{Y}}\) (i.e., \(\Lambda_0\) is a vertex lattice of type \(2h\)), the set \(\mathcal{Z}^{\rm loc}(\Lambda) = \mathcal{Y}^{\rm loc}(\Lambda^\sharp)\) is a singleton, corresponding to a discrete point in the RZ space. 
    \item The intersection $\mathcal{Z}^{\rm loc}(\Lambda) \cap \mathcal{Z}^{\rm loc}(\Lambda')$ (resp. $\mathcal{Y}^{\rm loc}(\Lambda^\sharp) \cap \mathcal{Y}^{\rm loc}(\Lambda^{\prime \sharp)}$) is non-empty if and only if $\Lambda'' = \Lambda + \Lambda'$ (resp. $\Lambda''=\Lambda \cap \Lambda'$) is a vertex lattice; in which case we have $\mathcal{Z}^{\rm loc}(\Lambda) \cap \mathcal{Z}^{\rm loc}(\Lambda') = \mathcal{Z}^{\rm loc}(\Lambda'')$ (resp. $\mathcal{Y}^{\rm loc}(\Lambda^\sharp) \cap \mathcal{Y}^{\rm loc}(\Lambda^{\prime \sharp}) = \mathcal{Y}^{\rm loc}(\Lambda^{\prime \prime \sharp})$).
\end{enumerate}
\end{Theorem}
\begin{proof}
For the stratification (\ref{BT_strataLoc}) and the relations (1)-(4) see \cite[Theorem 5.5]{HLS2}. For (5), assume $\mathcal{Z}^{\rm loc}(\Lambda) \cap \mathcal{Z}^{\rm loc}(\Lambda') \neq \varnothing$ and let $ M \in \mathcal{Z}^{\rm loc}(\Lambda) \cap \mathcal{Z}^{\rm loc}(\Lambda') (\kappa) $. Then, $ \Lambda+\Lambda' \subset \Lambda(M) $ where $\Lambda(M)$ is the maximal vertex lattice contained in $M$ given by Proposition \ref{prop 51}. Then, $ \Lambda+\Lambda' \subset \Lambda(M) \subset (\Lambda(M))^{\sharp} \subset (\Lambda+\Lambda' )^{\sharp} $ and so $\Lambda'' = \Lambda + \Lambda'$ is a vertex lattice. A similar argument works for the $\mathcal{Y}^{\rm loc}$-strata.
\end{proof}

\begin{Theorem}\label{Thm BT}
The Bruhat-Tits stratification of the reduced basic locus of splitting {\rm RZ} space $\calN_{n, {\rm red}}^\spl$ is
\begin{equation}\label{BTstr}
      \calN_{n, {\rm red}}^\spl = \left( \bigcup_{\Lambda_1 \in L_{\mathcal{Z}} }\mathcal{Z}(\Lambda_1) \right) \cup \left( \bigcup_{\Lambda_2 \in L_{\mathcal{Y}}} \mathcal{Y}(\Lambda_2^\sharp) \right).
  \end{equation}
\begin{enumerate}
  \item  
   These strata satisfy the following inclusion relations:
  \begin{itemize}
    \item[(i)] For any $\Lambda_1, \Lambda_2 \in L_{\mathcal{Z}}$ of type greater than $2h$, if 
    $  \Lambda_1 \subseteq \Lambda_2$ then $\mathcal{Z}(\Lambda_2) \subseteq \mathcal{Z}(\Lambda_1)$.    
    \item[(ii)] For any $\Lambda_1, \Lambda_2 \in  L_{\mathcal{Y}}$ of type less than $2h$, if  
    $  \Lambda_1 \subseteq \Lambda_2$ then $\mathcal{Y}(\Lambda^{\sharp}_1) \subseteq \mathcal{Y}(\Lambda^{\sharp}_2)$.  

    \item[(iii)] For any $\Lambda_1\in L_{\mathcal{Z}}$ of type greater than $2h$, $\Lambda_2 \in  L_{\mathcal{Y}}$ of type less than $2h$, $  \Lambda_1 \subseteq \Lambda_2$ if and only if the intersection $\mathcal{Z}(\Lambda_1) \cap \mathcal{Y}(\Lambda_2^\sharp)$ is non-empty.
  \end{itemize}

  \item In the following, assume that $\Lambda, \Lambda'$ are vertex lattices of type $ 2t$ with $t \neq h$, and $\Lambda_0, \Lambda'_0$ are vertex lattices of type $2t$ with $ t=h$. 
  \begin{itemize}
    \item[(i)] The intersection $\mathcal{Z}(\Lambda) \cap \mathcal{Z}(\Lambda')$ (resp. $\mathcal{Y}(\Lambda^\sharp) \cap \mathcal{Y}(\Lambda^{\prime \sharp)}$) is non-empty if and only if $\Lambda'' = \Lambda + \Lambda'$ (resp. $\Lambda''=\Lambda \cap \Lambda'$) is a vertex lattice; in which case we have $\mathcal{Z}(\Lambda) \cap \mathcal{Z}(\Lambda') = \mathcal{Z}(\Lambda'')$ (resp. $\mathcal{Y}(\Lambda^\sharp) \cap \mathcal{Y}(\Lambda^{\prime \sharp}) = \mathcal{Y}(\Lambda^{\prime \prime \sharp})$).

    \item[(ii)] The intersection $\mathcal{Z}(\Lambda_0) \cap \mathcal{Z}(\Lambda_0')$ (or $\mathcal{Y}(\Lambda_0^\sharp) \cap \mathcal{Y}(\Lambda_0^{\prime \sharp})$) is always empty if $\Lambda_0 \ne \Lambda'_0$.
    
    \item[(iii)] The intersection $\mathcal{Z}(\Lambda) \cap \mathcal{Z}(\Lambda_0)$ (resp.  $\mathcal{Y}(\Lambda^\sharp) \cap \mathcal{Y}(\Lambda_0^\sharp)$) is non-empty if and only if $\Lambda \subset \Lambda_0$ (resp. $\Lambda_0 \subset \Lambda$), in which case $\mathcal{Z}(\Lambda) \cap \mathcal{Z}(\Lambda_0)$ (resp.  $\mathcal{Y}(\Lambda^\sharp) \cap \mathcal{Y}(\Lambda_0^\sharp)$) is isomorphic to $\PP_k^{t-h-1}\times \PP_k^{2h}$ (resp. $\PP_k^{h-t-1}\times \PP_k^{n-2h-1}$).
    
    \item[(iv)] The BT-strata $\mathcal{Z}(\Lambda_0)$ and $\mathcal{Y}(\Lambda_0^\sharp)$ are each isomorphic to the blow-up of $\PP_k^{n-1}$ along $\PP_k^{2h-1}$ which we denote by $\rm{Bl}_{\PP_k^{2h-1}}(\PP_k^{n-1})$.
   \end{itemize}
\end{enumerate}
\end{Theorem}
\begin{proof}
The stratification (\ref{BTstr}) and the inclusion relations (1.i) and (1.ii) follow from Theorem \ref{BTStratLoc} together with the constructions of the RZ space $\calN_{n, {\rm red}}^\spl$ and the strict transforms $ \mathcal{Z}(\Lambda_1),\, \mathcal{Y}(\Lambda_2^\sharp) $.

(1.iii). If $\mathcal{Z}(\Lambda_1) \cap \mathcal{Y}(\Lambda_2^\sharp) \neq \varnothing$ then $\mathcal{Z}^{\rm loc}(\Lambda_1) \cap \mathcal{Y}^{\rm loc}(\Lambda_2^\sharp) \neq \varnothing$ and so $\Lambda_1 \subseteq \Lambda_2$ by Theorem \ref{BTStratLoc}. Conversely, suppose $\Lambda_1\subset \Lambda_2$ which gives $\mathcal{Z}^{\rm loc}(\Lambda_1) \cap \mathcal{Y}^{\rm loc}(\Lambda_2^\sharp)$ is non-empty. This in turn gives $\M^{\loc, [2h]}_n (2t_1)\cap \M^{\loc, [2h]}_n (2t_2) \neq \varnothing$. By calculating the strict transform and using the local model diagram we can conclude that $\mathcal{Z}(\Lambda_1) \cap \mathcal{Y}(\Lambda_2^\sharp) \neq \varnothing$.

(2.i). First, we show the statement for the $\mathcal{Z}^{\rm Kra}$ and $\mathcal{Y}^{\rm Kra}$-strata. For the $\mathcal{Z}^{\rm Kra}$-strata, if we assume that $\Lambda'' = \Lambda+\Lambda'$ is a vertex lattice, then we can easily see that $\mathcal{Z}^{\rm Kra}(\Lambda) \cap \mathcal{Z}^{\rm Kra}(\Lambda') = \mathcal{Z}^{\rm Kra}(\Lambda'')$ by construction. On the other hand, if we assume that $\mathcal{Z}^{\rm Kra}(\Lambda) \cap \mathcal{Z}^{\rm Kra}(\Lambda')$ is nonempty and pick $ (M,M',M'') \in  \mathcal{Z}^{\rm Kra}(\Lambda) \cap \mathcal{Z}^{\rm Kra}(\Lambda')(\kappa)$, then $\Lambda_1 \supset \Lambda + \Lambda'$ where $ \Lambda_1$ is the maximal vertex lattice contained in $M$ from Proposition \ref{prop 51}. Then $\Lambda + \Lambda' \subset \Lambda_1 \subset \Lambda_1^\sharp \subset \Lambda^\sharp \cap (\Lambda')^\sharp = (\Lambda + \Lambda')^\sharp$. Similarly, we have $\pi(\Lambda+\Lambda')^\sharp\subset (\Lambda+\Lambda')$. Hence, $\Lambda''=\Lambda + \Lambda'$ is a vertex lattice. For $\mathcal{Y}^{\rm Kra}$-strata, note that $ (M,M',M'') \in  \mathcal{Y}^{\rm Kra}(\Lambda^{\sharp}) \cap \mathcal{Y}^{\rm Kra}((\Lambda')^{\sharp})(\kappa)$ gives  $\Lambda_2\subset \Lambda\cap \Lambda'\subset (\Lambda^\sharp+\Lambda')^\sharp\subset \Lambda_2^\sharp$ by the minimality of $\Lambda_2$ from Proposition \ref{prop 51}. Then, arguing as in the case of the $\mathcal{Z}^{\rm Kra}$-strata, we obtain the desired result.

Now, using the above, 
consider the closed immersion $i: \mathcal{N}^\spl_n \xrightarrow{} \mathcal{N}_n^{\rm Kra} $. The base change along $i$ gives:
\begin{flalign*}
\mathcal{Z}(\Lambda'') &= \mathcal{Z}^{\rm Kra}(\Lambda'')\times_{\mathcal{N}_n^{\rm Kra}} \mathcal{N}^\spl_n  =  \left(\mathcal{Z}^{\rm Kra}(\Lambda')\times_{\mathcal{N}_n^{\rm Kra}}\mathcal{Z}^{\rm Kra}(\Lambda)\right)\times_{\mathcal{N}_n^{\rm Kra}} \mathcal{N}^\spl_n \\
&=  \left(\mathcal{Z}^{\rm Kra}(\Lambda')\times_{\mathcal{N}_n^{\rm Kra}}\mathcal{N}^\spl_n\right)\times_{\mathcal{N}^\spl_n} \left(\mathcal{Z}^{\rm Kra}(\Lambda)\times_{\mathcal{N}_n^{\rm Kra}}\mathcal{N}^\spl_n\right) \\
&= \mathcal{Z}(\Lambda)\cap \mathcal{Z}(\Lambda').   
\end{flalign*}
A similar argument works for the $\calY$-strata.

(2.ii), (2.iv). Similar to Proposition \ref{prop KraNexc}, each point of $ \mathcal{Z}(\Lambda_0)$ and $\mathcal{Y}(\Lambda_0^\sharp)$ is \'etale locally isomorphic to a point in the exceptional divisor $\Exc$. Then (2.ii) follows directly from the definitions and the constructions of $ \mathcal{Z}(\Lambda_0)$ and $ \mathcal{Y}(\Lambda_0^\sharp)$; (2.iv) follows from Theorem \ref{thm 3.6}.

(2.iii). Using the local model diagram, it amounts to calculating the exceptional locus of $\M^{\loc, [2h]}_n (2t) $ for the corresponding strata. Now the result follows from Propositions \ref{prop Zbl} and \ref{prop Ybl}.
\end{proof}

\section{Appendix}
The goal of this section is to discuss the relationship between the splitting-model variants of \cite{HLS1} and \cite{ZacZhao2}. We have already defined the naive splitting model $\M^{\nspl,[2h]}_n$, and the splitting model $\M^{\spl,[2h]}_n$ which is the scheme-theoretic closure of the generic fiber in  $\M^{\nspl,[2h]}_n$. These models were first studied in \cite{HLS1}.
We refer to \S\ref{RSpl} for their definitions and basic properties. Next, we briefly recall the splitting models defined in \cite{ZacZhao2} and their basic properties; see \S \ref{RSpl} for any undefined terms.

\begin{Definition}
{\rm
The splitting model $\mathscr{M}^{\spl,[2h]}_n$ of \cite{ZacZhao2} is a projective scheme over $\Spec O_F$ representing the functor that sends each $O_F$-algebra $R$ to the set of subsheaves 
\[
\begin{array}{l}
   \calF_i \subset  \Lambda_{i,R}, \quad ~\text{where}~ i\in [2h]  \\
   \calG_j \subset  \calF_j, \quad ~\text{where}~ j\in \{-h\}+n\ZZ
\end{array}
\]
such that
\begin{itemize}
    \item For all $i\in [2h], j\in \{-h\}+n\ZZ$, $\calF_i$ (resp. $\calG_j$) as $O_F\otimes R$-modules are Zariski locally on $R$ direct summands of rank $n$ (resp. rank $1$).
    \item For all $i\in [2h]$, $(\calF_i)\in \M^{[2h],\wedge}_n\otimes R$.
    \item (Splitting condition)  For all $j\in \{-h\}+n\ZZ$,
    \[
    (\Pi+\pi) \calF_{j} \subset \calG_{j}, \quad  (\Pi-\pi) \calG_{j} = (0).
    \]
\end{itemize}
}
\end{Definition}
There is a projective morphism  $\tau' : \mathscr{M}^{\spl,[2h]}_n \rightarrow \M^{\loc,[2h]}_n$ which is given by $(\calF_i, \calG_{j}) \mapsto (\calF_i)$ on $R$-valued points. Similar to the porjective morphism $\tau$, the morphism $\tau'$ induces an isomorphism 
\[
\mathscr{M}^{\spl,[2h]}_n\setminus (\tau')^{-1}(*) \cong  \M^{\loc,[2h]}_n\setminus\{*\},
\]
where $*$ is the worst point in the local model. From the above construction, we deduce that there is a forgetful projective morphism $\psi:  \M^{\spl,[2h]}_n \rightarrow \mathscr{M}^{\spl,[2h]}_n$ such that the following diagram is commutative: 
\begin{equation}\label{Rlt601}
\begin{tikzcd}
\M^{\spl,[2h]}_n \arrow{r}{\psi} \arrow[swap]{dr}{\tau} & 
\mathscr{M}^{\spl,[2h]}_n \arrow{d}{\tau'} \\
& \M^{\loc,[2h]}_n 
\end{tikzcd}.
\end{equation}
The projective scheme $\mathscr{M}^{\spl,[2h]}_n$ is flat, normal, Cohen-Macaulay and its special fiber is reduced (see \cite[Theorem 5.0.1]{ZacZhao2}). Denote by $\mathscr{M}^{\bl,[2h]}_n$ the blow-up of $\mathscr{M}^{\spl,[2h]}_n$ along $(\tau')^{-1}(*) $. Let $ r_1: \mathscr{M}^{\bl,[2h]}_n \rightarrow \mathscr{M}^{\spl,[2h]}_n $ be the blow-up morphism. The projective scheme $\mathscr{M}^{\bl,[2h]}_n$ is regular and has special fiber a reduced divisor with normal crossings. We refer to  \cite[\S 7]{ZacZhao2} for more details on $\mathscr{M}^{\bl,[2h]}_n$.
\begin{Proposition}
The projective scheme $\M^{\spl,[2h]}_n$ is equal to $ \mathscr{M}^{\bl,[2h]}_n$.
\end{Proposition}
\begin{proof}
By Theorem \ref{thm 3.6} we have $\M^{\spl,[2h]}_n = \M^{\bl,[2h]}_n$ where $\M^{\bl,[2h]}_n$ is the blow-up of $\M^{\loc,[2h]}_n$ along $*$. (Below, we use these terms interchangeably.) Set $ r = \tau' \circ r_1: \mathscr{M}^{\bl,[2h]}_n \rightarrow \M^{\loc,[2h]}_n$. We have $ r^{-1}(*) = r_1^{-1}((\tau')^{-1}(*))$ and $ r^{-1}(*)$ is an effective Cartier divisor since $r_1$ is the blow-up morphism of $\mathscr{M}^{\spl,[2h]}_n$ along $(\tau')^{-1}(*)$. Thus, from the universal property of the blow-up there is a unique morphism $g:  \mathscr{M}^{\bl,[2h]}_n \xrightarrow{} \M^{\bl,[2h]}_n $ such that $ r= \tau \circ g$ and the following diagram is commutative:
\begin{equation}\label{Rlt602}
\begin{tikzcd}
\mathscr{M}^{\bl,[2h]}_n \arrow{r}{g} \arrow{d}[swap]{r_1} & \M^{\bl,[2h]}_n \arrow{d}{\tau} \\
\mathscr{M}^{\spl,[2h]}_n \arrow{r}[swap]{\tau'}            &  \M^{\loc,[2h]}_n
\end{tikzcd}.
\end{equation}

Now observe from (\ref{Rlt601}) that $ \psi^{-1}((\tau')^{-1}(*)) = (\tau'\circ \psi )^{-1}(*) = \tau^{-1}(*)$ is an effective Cartier divisor and so, as above, there is a unique morphism $ h: \M^{\bl,[2h]}_n \rightarrow \mathscr{M}^{\bl,[2h]}_n $  such that $ \psi = r_1 \circ h$, i.e., the following diagram is commutative:
\begin{equation}\label{Rlt603}
\begin{tikzcd}
\M^{\bl,[2h]}_n \arrow{r}{h} \arrow[swap]{dr}{\psi} & 
\mathscr{M}^{\bl,[2h]}_n \arrow{d}{r_1} \\
& \mathscr{M}^{\spl,[2h]}_n 
\end{tikzcd}.
\end{equation}
We claim that $h$ is a section of $g$. To see this consider 
\[
\tau\circ(g \circ h) = (\tau \circ g ) \circ h \overset{(\ref{Rlt602})}{=}(\tau' \circ r_1) \circ h = \tau' \circ (r_1 \circ h)  \overset{(\ref{Rlt603})}{=} \tau' \circ \psi  \overset{(\ref{Rlt601})}{=} \tau.
\]
By the universal property of the blow-up $ \tau: \M^{\bl,[2h]}_n \rightarrow \M^{\loc,[2h]}_n $, the only morphism $ g \circ h : \M^{\bl,[2h]}_n  \rightarrow\M^{\bl,[2h]}_n $ with $ \tau \circ ( g \circ h ) = \tau $ is $  g \circ h = \text{id}_{\M^{\bl,[2h]}_n }$. Thus, the claim follows.

From the above we see that the projective (hence separated) morphism $g:  \mathscr{M}^{\bl,[2h]}_n \xrightarrow{} \M^{\bl,[2h]}_n $ has a section $h$. Thus, $ h: \M^{\bl,[2h]}_n \rightarrow \mathscr{M}^{\bl,[2h]}_n $ is a closed immersion. Since $\M^{\bl,[2h]}_n ,\,\mathscr{M}^{\bl,[2h]}_n$ are integral flat schemes over $O_F$ with the same relative dimension we conclude that $\M^{\bl,[2h]}_n =\mathscr{M}^{\bl,[2h]}_n$.
\end{proof}
\begin{Remark}
{\rm
We want to mention another route that one can follow to obtain the BT stratification of the RZ space $\calN_{n, {\rm red}}^\spl$ of $\M^{\spl,[2h]}_n $ given in Theorem \ref{Thm BT}. From the above proposition we have $\M^{\spl,[2h]}_n =\mathscr{M}^{\bl,[2h]}_n$. In \cite{ZacZhao3}, we construct the BT-stratification of the RZ space $\mathscr{N}_{n, {\rm red}}$ of $\mathscr{M}^{\spl,[2h]}_n$. From this, we can define the BT-strata of $\calN_{n, {\rm red}}^\bl$ as certain strict transforms of the corresponding BT strata of $\mathscr{N}_{n, {\rm red}}$ as in the construction of the $\calZ$-strata (resp. $\calY$-strata). Then using the corresponding strata splitting models $\calZ^\spl(\Lambda)$ (resp. $\calY^\spl(\Lambda^\sharp)$) constructed in \cite{ZacZhao3} one can recover Theorem \ref{Thm BT}.
}
\end{Remark}

\Addresses
\end{document}